\documentclass[12pt,reqno]{amsart}
\usepackage{amssymb}
\usepackage{diagrams}
\usepackage{graphicx}
\usepackage{amsmath}
\usepackage{a4wide}
\usepackage{version}
\usepackage{mathrsfs}
\usepackage{tikz}
\usepackage{tikz-cd}
\usepackage{mathtools}
\usetikzlibrary{arrows,decorations.markings, matrix}
\usepackage{hyperref}
\usepackage{float}

\usepackage[all]{xy}

\numberwithin{equation}{section}

\newtheorem{thm}{Theorem}[subsection]
\newtheorem{lem}[thm]{Lemma}
\newtheorem{prop}[thm]{Proposition}
\newtheorem{cor}[thm]{Corollary}

\theoremstyle{definition}
\newtheorem{definition}[thm]{Definition}

\newtheorem{rmk}[thm]{Remark}
\newtheorem{conj}[thm]{Conjecture}

\newcommand{\Z}{\mathbb{Z}}
\renewcommand{\k}{\mathbf{k}}

\newcommand{\stab}{\operatorname{stab}}
\newcommand{\Pf}{\noindent {\it Proof}}
\newcommand{\id}{\operatorname{id}}

\newcommand{\ov}{\overline}

\newcommand{\WW}{{\mathcal W}}
\newcommand{\EE}{{\mathcal E}}

\newcommand{\HH}{{\mathcal H}}

\renewcommand{\SS}{{\mathcal S}}

\newcommand{\MF}{\operatorname{MF}}
\newcommand{\Hom}{\operatorname{Hom}}

\newcommand{\Ext}{\operatorname{Ext}}

\renewcommand{\a}{\alpha}
\renewcommand{\b}{\beta}

\newcommand{\la}{\lambda}
\renewcommand{\th}{\theta}
\newcommand{\C}{{\mathbb C}}

\newcommand{\Ga}{\Gamma}

\newcommand{\wt}{\widetilde}
\newcommand{\ot}{\otimes}

\newcommand{\sub}{\subset}
\newcommand{\ed}{\qed\vspace{3mm}}

\renewcommand{\k}{\mathbf{k}}

\renewcommand{\mod}{\operatorname{mod}}

\newcommand{\OO}{{\mathcal O}}
\newcommand{\RR}{{\mathcal R}}

\newcommand{\rot}{\operatorname{rot}}
\newcommand{\eend}{\operatorname{end}}

\newcommand{\DD}{{\mathcal D}}

\newcommand{\II}{{\mathcal I}}

\newcommand{\BB}{{\mathcal B}}

\newcommand{\G}{{\mathbb G}}

\newcommand{\hra}{\hookrightarrow}
\newcommand{\lan}{\langle}
\newcommand{\ran}{\rangle}

\newcommand{\CC}{{\mathcal C}}

\newcommand{\si}{\sigma}

\newcommand{\ga}{\gamma}
\newcommand{\de}{\delta}
\newcommand{\eps}{\epsilon}

\renewcommand{\ker}{\operatorname{ker}}

\newcommand{\A}{{\mathbb A}}

\newcommand{\bideg}{\operatorname{bideg}}
\begin{document}
\title[HMS for chain polynomials]{On homological mirror symmetry for chain type polynomials}
\author{Umut Varolgunes}
\address{(UV) Stanford University}
\author{Alexander Polishchuk}
\address{(AP) University of Oregon, National Research University Higher School of Economics, and
Korea Institute for Advanced Study}

\begin{abstract}
	We consider Takahashi's categorical interpretation of the Berglund-Hubsch mirror symmetry conjecture for invertible polynomials in the case of chain polynomials. Our strategy is based on a stronger claim that the relevant categories satisfy a recursion of directed $A_{\infty}$-categories, which may be of independent interest. We give a full proof of this claim on the B-side. On the A-side we give a detailed sketch of an argument, which falls short of a full proof because of certain missing foundational results in Fukaya-Seidel categories, most notably a generation statement.
\end{abstract}
\maketitle


\section{Introduction}

Recall that a polynomial $w\in\C[x_1,\ldots,x_n]$ is called {\it  invertible} if 
$$w=\sum_{i=1}^n c_i\prod_{j=1}^n x_j^{a_{ij}}$$
for $c_i\in \C^*$ and a nondegenerate integer matrix $A=(a_{ij})$ and $w$ has an isolated critical point at the origin.
Such a polynomial is weighted homogeneous for a canonical system of weights, which is uniquely determined by requiring the weight of the action on $w$ to be $det(A)$. 
Rescaling the variables one can make all $c_i=1$.

For an invertible polynomial $w$ defined by the matrix $A$, the {\it dual} invertible  polynomial
$w^\vee$ is defined by the transposed matrix $A^t$.

Invertible polynomials can be classified by an elementary argument \cite{Kreuzer}. Every invertible polynomial is the sum of atomic ones in different sets of
variables. The atomic invertible polynomials are of the following three types:
\begin{itemize}
\item {\it Fermat} type: $x_1^{a_1}$, 
\item {\it chain} type: $x_1^{a_1}x_2+x_2^{a_2}x_3+\ldots +x_{n-1}^{a_{n-1}}x_n+x_n^{a_n}$, 
\item {\it loop} type: $x_1^{a_1}x_2+x_2^{a_2}x_3+\ldots +x_{n-1}^{a_{n-1}}x_n+x_n^{a_n}x_1$, 
\end{itemize}
where $n>1$ and all $a_i>1$. In fact, we will think of the Fermat polynomials as chain type polynomials with $n=1$.

The homological mirror symmetry conjecture for invertible polynomials states for
an invertible $w$ and its dual $w^\vee$ that there is an equivalence of triangulated categories
\begin{align}\label{conj-hms} D(F(w))\simeq D(\MF_{\Ga}(w^\vee))\end{align}
between the derived Fukaya-Seidel category of $w$ and the derived category of maximally graded matrix
factorizations of $w^\vee$ (see Conjecture 21 from \cite{Miami}, which seems to have been inspired by Conjecture 7.6 from \cite{Takahashi}). To be precise, here we use the Fukaya-Seidel category as constructed in Seidel's very first paper in the subject \cite{Seidelmutation}.

In the present work we consider this conjecture in the case of chain polynomials. Note that for the chain polynomial
$$p_a:=x_1^{a_1}x_2+x_2^{a_2}x_3+\ldots +x_{n-1}^{a_{n-1}}x_n+x_n^{a_n}$$
depending on the vector $a=(a_1,\ldots,a_n)\in \Z_{>1}^n$, the dual polynomial is
$p_{a^\vee}$ where 
$a^\vee=(a_n,\ldots,a_1)$. Let us mention that for chain polynomials in one and two variables, complete proofs of the conjecture exist (see \cite{Ueda} for the $n=1$ and $a=(2,a_2)$ cases, and \cite{Habermann} for the general $n=2$ case).



Our strategy is based on a recursive computation of the relevant categories which may be of independent interest. It is known that the categories on both sides admit full exceptional collections. On the A-side we use the Morsification and distinguished basis introduced in \cite{Seifert}, while on the B-side we use the full exceptional collection constructed by 
Aramaki and Takahashi in \cite{Aramaki} (to which we often refer as {\it AT-collection}). That these two full exceptional collections should correspond to each other under a homological mirror functor was conjectured in \cite{Seifert}. Thus, we can reformulate the conjecture as an equivalence
of the corresponding directed $A_{\infty}$-categories (with objects given by the specified full exceptional collections), 
which we denote as $F(p_a)$ and $AT(a^\vee)$.

\subsection{A recursion for directed $A_{\infty}$-categories and the Main Claim}
We say that two directed $A_{\infty}$-categories are equivalent if there is an $A_{\infty}$ quasi-isomorphism between them which preserves the ordering of the objects. 

Our recursion 
is based on the following operation 
for directed $A_\infty$-categories. 
Given 
a directed $A_\infty$-category $\CC$ with objects $e=(E_1,\ldots,E_n)$  and
a number $N>n$, we construct a new directed $A_\infty$-category $\CC^+$ with $N$ objects $e^+$, as follows. 

\begin{itemize}
\item Extend $e$ to a helix inside $Tw(\mathcal{C})$ and take the segment 
$f$ of length $N$ in this helix ending with $E_1$.
\item Note that $f$ is no longer an exceptional collection in general (it can even have repeated elements). We define $\mathcal{C}'$ as the directed $A_\infty$-category defined by the directed $A_\infty$-subcategory of $f$ (keeping track of only morphisms from left to right in the order of the helix).
\item Inside $Tw(\mathcal{C}'),$ we consider the right dual exceptional
collection $e^+$ and define $\CC^+$ to be the corresponding directed $A_\infty$-category. 
\end{itemize}

We will loosely say that a directed $A_\infty$-category is obtained from $\CC$ by the recursion $\mathcal{R}$ with number $N$ if it is equivalent (as a directed $A_\infty$-category) to $\CC^+$ described above.

For any directed $A_\infty$-category and an $m$-tuple of integers $\sigma=(\sigma_1,\ldots,\sigma_m)$, we can define the $\sigma$-shifted directed $A_\infty$-category by changing the grading of morphism spaces by $\sigma_i-\sigma_j$. If one directed $A_\infty$-category is equivalent to a shifted version of another, we say that these two are equivalent {\it up to shifts}. We say that a directed $A_\infty$-category is obtained from $\CC$ by the recursion $\mathcal{R}$ with number $N$ 
{\it up to shifts} if it is equivalent up to shifts to $\CC^+$ described above. Note that the application of $\mathcal{R}$ to directed $A_{\infty}$-categories equivalent up to shifts result in directed $A_{\infty}$-categories which are equivalent up to shifts.

Let us call the following our Main Claim for A- and B-sides. On the A-side we claim that
$F(p_{a_1,\ldots,a_n})$ 
is obtained from $F(p_{a_2,\ldots,a_n})$ 
by the the recursion $\RR$ up to shifts with $N=\mu(a_1,\ldots,a_n)$, the Milnor number of the singularity of $p_a$. On the B-side we claim that $AT(a_n,\ldots,a_1)$ 
is obtained by the recursion $\RR$ from $AT(a_n,\ldots,a_2)$ 
again with $N=\mu(a_1,\ldots,a_n)$, up to shifts. 

In fact, we make this claim starting from $n=0$, where the corresponding $A_{\infty}$ categories on both sides 
are the same: the category $\CC_{\varnothing}$ with one object $E$ and $Hom(E,E)=\mathbb{Z}$ 
concentrated in degree $0$. Therefore, our Main Claim for A- and B-sides lead to a proof of the homological mirror symmetry conjecture for the chain polynomials.

We prove the Main Claim for B-side fully. We are also able to compute the relevant shifts. On the A-side we give a detailed sketch of an argument that we believe the reader will find quite convincing.  We do not attempt to compute the shifts. A full proof on the A-side awaits the development of a couple of foundational results about Fukaya-Seidel categories of tame Landau-Ginzburg models. We explain these results in Section \ref{ss-fukaya-seidel}, specifically see Remarks \ref{rmk-C0-bounds} and \ref{rmk-generation}. There is a less major point in which our argument falls short of a full proof, which is explained in Remark \ref{rmk-homotopy-method}.

\begin{rmk}
The recursion $\mathcal{R}$ is bound to be related to the recursion of Seifert matrices that was used in \cite{Seifert}, but we do not know exactly how.
\end{rmk}



\subsection{An equivariant equivalence}\label{sssymetry} 
In this section, we introduce some notation that will be used later and also discuss  the  symmetries of both sides. How the symmetries on both sides correspond to each other    is an important guiding principle for our strategy. Of course symmetries played an important since inception of Berglund-Hubsch-Henningson mirror symmetry conjecture \cite{Berglund}, \cite{Henningson}.

We define the group of symmetries of $p_a$ to be 
\begin{align}\Gamma_a:=\{(\lambda_1,\ldots, \lambda_n,\lambda)\mid \lambda_1^{a_1}\lambda_2=\ldots= \lambda_{n-1}^{a_{n-1}}\lambda_n=\lambda_{n}^{a_{n}}=\lambda\}\subset (\mathbb{C}^*)^{n+1}.\end{align} It is easy to see that $\Gamma_a$ is a graph over $\{\lambda_1^{d(a)}=\lambda^{\mu(a_2,\ldots , a_n)}\}\subset (\mathbb{C}^*)^2.$

We also define: \begin{align}\Gamma_a^0:=\{(\lambda_1,\ldots, \lambda_n)\mid \lambda_1^{a_1}\lambda_2=\ldots= \lambda_{n-1}^{a_{n-1}}\lambda_n=\lambda_{n}^{a_{n}}\}\subset (\mathbb{C}^*)^{n+1},\end{align} which is isomorphic to the subgroup of $\Gamma_a$ given by $\lambda=1$. In what follows we denote the generator of $\Gamma_a^0$ with $\lambda_1=e^{\frac{2\pi i}{d(a)}}$ by $\phi_a$. By an abuse of notation we use $\phi_a$ also for the symplectomorphism of $\mathbb{C}^n$ given by the action of $\phi_a\in \Gamma_a^0$.

Now consider $\hat{\Gamma}_a^0$ the group of graded symplectomorphisms of $\mathbb{C}^n$ whose underlying symplectomorphism is given by the action of an element of $\Gamma_a^0$. There is a short exact sequence of groups:

$$0\to \mathbb{Z}\to \hat{\Gamma}_a^0 \to\Gamma_a^0\to 0.$$

The group $\hat{\Gamma}_a^0$ naturally acts on $D(F(p_a))$, with the image of $1$ in $\hat{\Gamma}_a^0$ acting as the shift $[1]$. We will not use this action except for stating Conjecture \ref{conj-hms-eq} and the remark proceeding it, so we omit the details.

Let us also consider the Pontrjagin dual of $\Gamma_a$, 
$$L_a:=Hom(\Gamma_a,\mathbb{C}^*),$$
which we identify with the abelian group with the generators $\ov{x}_1,\ldots,\ov{x}_n,\ov{p}$ and the defining relations
$$a_1\ov{x}_1+\ov{x}_2=\ldots=a_{n-1}\ov{x}_{n-1}+\ov{x}_n=a_n\ov{x}_n=\ov{p}.$$ 
The action of $\Gamma_a$ on $\mathbb{C}^n$ provides $\mathbb{C}[x_1,\ldots ,x_n]$ with an $L_a$-grading, so that
$x_i$ has degree $\ov{x}_i$ and $p_a$ has degree $\ov{p}$ (this is the maximal grading for which $p_a$ is homogenous). As a result, $L_a$ canonically acts on $D(\MF_{\Ga}(p_{a}))$ (see \cite[Sec.\ 2]{Aramaki}). 
In fact, it is more convenient to consider a $\mathbb{Z}/2$ extension of $L_a$ called $\wt{L}_a$, which has an
additional generator $T$ acting on $D(\MF_{\Ga}(p_{a}))$ by a shift. The group $\wt{L}_a$ is generated by
two elements: $T$ and $$\tau=(-1)^n\ov{x}_1$$ subject to the single relation
\begin{equation}\label{tau-T-relation}
d(a)\tau=(-1)^n2(d(a)-\mu(a))T,
\end{equation}
where $\mu(a)=\mu(a_1,\ldots,a_n)=a_1\ldots a_n-a_2\ldots a_n+a_3\ldots a_n-\ldots$ is the Milnor number, and $d(a)=a_1\ldots a_n$
(see Sec.\ \ref{AT-basic-sec}).

Finally, we set $$L^0_a:=Hom(\Gamma_a^0,\mathbb{C}^*),$$ and note the existence of the short exact sequence
$$0\to \mathbb{Z}\to \wt{L}_a \to L_a^0\to 0.$$

It is well known that $\Gamma_a^0$ is isomorphic to $L_{a^\vee}^0$, but  the following extension appears to be new.

\begin{prop}\label{prop-graded-lift}
$\hat{\Gamma}_a^0$ is isomorphic to $\wt{L}_{a^\vee}$ as an extension of $\Gamma_a^0=L_{a^\vee}^0$ by $\mathbb{Z}$. Under this isomorphism, the element $\tau\in \wt{L}_{a^\vee}$ corresponds to some explicit graded lift 
$\wt{\phi}_a$ of $\phi_a$.
\end{prop}
\begin{proof}
Let us take the graded lift $\tilde{\phi}_a$ of $\phi_a$ that comes from it being the time $1$ map of the flow $$(x_1,\ldots ,x_n)\to (e^{\frac{2\pi t i}{d(a)}}x_1,\ldots ,e^{(-1)^{n-1}\frac{2\pi t i}{a_n}}x_n).$$ 
We need to check that the generators $\tilde{\phi}_a$ and $1\in \Z$ of $\hat{\Gamma}_a^0$ satisfy the same
relation as $\tau$ and $T$ in $\wt{L}_{a^\vee},$ i.e. Equation \eqref{tau-T-relation}. 
We know that $\tilde{\phi}_a^{d(a)}$ is a graded lift of the identity symplectomorphism. We need to compute how
it differs from the trivial graded lift. For this we choose the holomorphic volume form $dx_1\wedge\ldots \wedge dx_n$ and use the fact that the origin is fixed. Thus, we need to find the winding number of the path 
$$e^{\frac{4\pi t i(1-a_1+\ldots +(-1)^na_1\ldots a_{n-1})}{d(a)}}$$ 
as $t$ goes from $0$ to $d(a)$. This number is $(-1)^n2(d(a)-\mu(a^\vee))$, which gives the required relation.
\end{proof}

\begin{conj}\label{conj-hms-eq}
There is an HMS equivalence \begin{align}\label{conj-hms-chain} D(F(p_a))\simeq D(\MF_{L_{a^\vee}}(p_{a^\vee}))\end{align} equivariant with respect to the actions of $\hat{\Gamma}_a^0=\wt{L}_{a^\vee}$.
\end{conj}

\begin{rmk}
We already know that the generator $T\in \wt{L}_{a^\vee}$ acts on both sides as the shift functor. It is also known (Proposition 3.1 of \cite{Aramaki} and Lemma \ref{Serre-Bside-lem} below) that the second generator 
$\tau\in \wt{L}_{a^\vee}$ acts on the B-side by the autoequivalence satisfying 
\begin{align}\tau^{\mu(a)}=T^{N(a)}S^{-1},\end{align} 
where $S$ is the Serre functor and $N(a)$ is an explicit integer. If the expected relationship between monodromy and Serre functor on the A-side is true (see e.g. \cite{Jeffs} for a survey), then it can be shown that the action of $\wt{\phi}_a$ on the A-side satisfies the same property as well. Therefore, for the subgroup of $\wt{L}_{a^\vee}$ generated by $T$ and $\tau^{\mu(a)}$, the equivariance 
follows from this. 
The relation \eqref{tau-T-relation} shows that this subgroup is the entire $\wt{L}_{a^\vee}$ in the case 
when $\mu(a)$ and $d(a)$ are coprime, so in this case the equivariant conjecture follows from the non-equivariant one. 
The equivariant conjecture does not seem to follow from the non-equivariant one if $\mu(a)$ and $d(a)$ are not coprime.
\end{rmk}

%
%
%
%
%
%

We will use the perturbation $x_1+p_a$ whose distinguishing property is that the symmetry by $\Gamma_a^0$ persists to it in a way that we can explicitly describe. First, note that $x_1+p_a$ is equivariant with respect to
the order $\mu(a)$ cyclic subgroup of $\Gamma_a$ given by $\lambda_1=\lambda$. Let us denote the generator of this group with $\lambda=e^{\frac{2\pi i}{\mu(a)}}$ by $\psi_a$. 
Let us also define a symplectomorphism $\rho_{a,\epsilon}$ of $\mathbb{C}^n$ by lifting (using parallel transport)
the following diffeomorphism $\varphi_{\epsilon}$ of the base of $\epsilon x_1+p_a$ for all $|\epsilon|\leq 1$: 
it does nothing inside a disk which contains all the critical points; then starts rotating in an annulus in clockwise direction; the amount of rotation increases until it reaches $\frac{2\pi}{\mu(n)}$; and everything outside the annulus gets rotated by $\frac{2\pi}{\mu(n)}$ clockwise (see the right side of Figure \ref{monodromy}). 
Recall that the group $\Gamma_a^0$ is generated by the 
symplectomorphism $\phi_a$. The following proposition gives a symmetry of $x_1+p_a$, isotopic to $\phi_a$.

\begin{prop}\label{prop-pert-symplecto}
 $\rho_{a,\epsilon}\circ \psi_a$ is isotopic to $\phi_a$ through symplectomorphisms. 
\end{prop}
\begin{proof}
It is clear that  $\rho_{a,\epsilon}\circ \psi_a$ is isotopic to  $\rho_{a,0}\circ \psi_a$ through symplectomorphisms by considering a path in the complex plane from $\epsilon$ to $0$.

Now we note that the rotation of the base of $p_a$ by $\theta$ lifts to the symplectomorphism \begin{align*}(z_1,\ldots ,z_n)\mapsto (e^{i\theta w_1}z_1,\ldots ,e^{i\theta w_n}z_n),\end{align*} where 
$w_k=\frac{\mu(a_{k+1},\ldots,a_n)}{a_k\ldots a_n}$. 
Hence, we see that $\rho_{a,0}$ is isotopic through symplectomorphisms to 
\begin{align*}(z_1,\ldots ,z_n)\mapsto (e^{-\frac{2\pi i w_1}{\mu(a)}}z_1,\ldots ,e^{-\frac{2\pi i w_n}{\mu(a)}}z_n).\end{align*}
Recalling the definitions of $\psi_a$ and $\phi_a$, we see that the assertion follows from
 $$\frac{1}{\mu(a)}-\frac{\mu(a_2,\ldots,a_k)}{d(a)\mu(a)}=\frac{1}{d(a)}.$$
\end{proof}

\subsection{More details on the A-side}\label{ss-Aside-intro}
%
%

\begin{figure}
\includegraphics[width=0.5\textwidth]{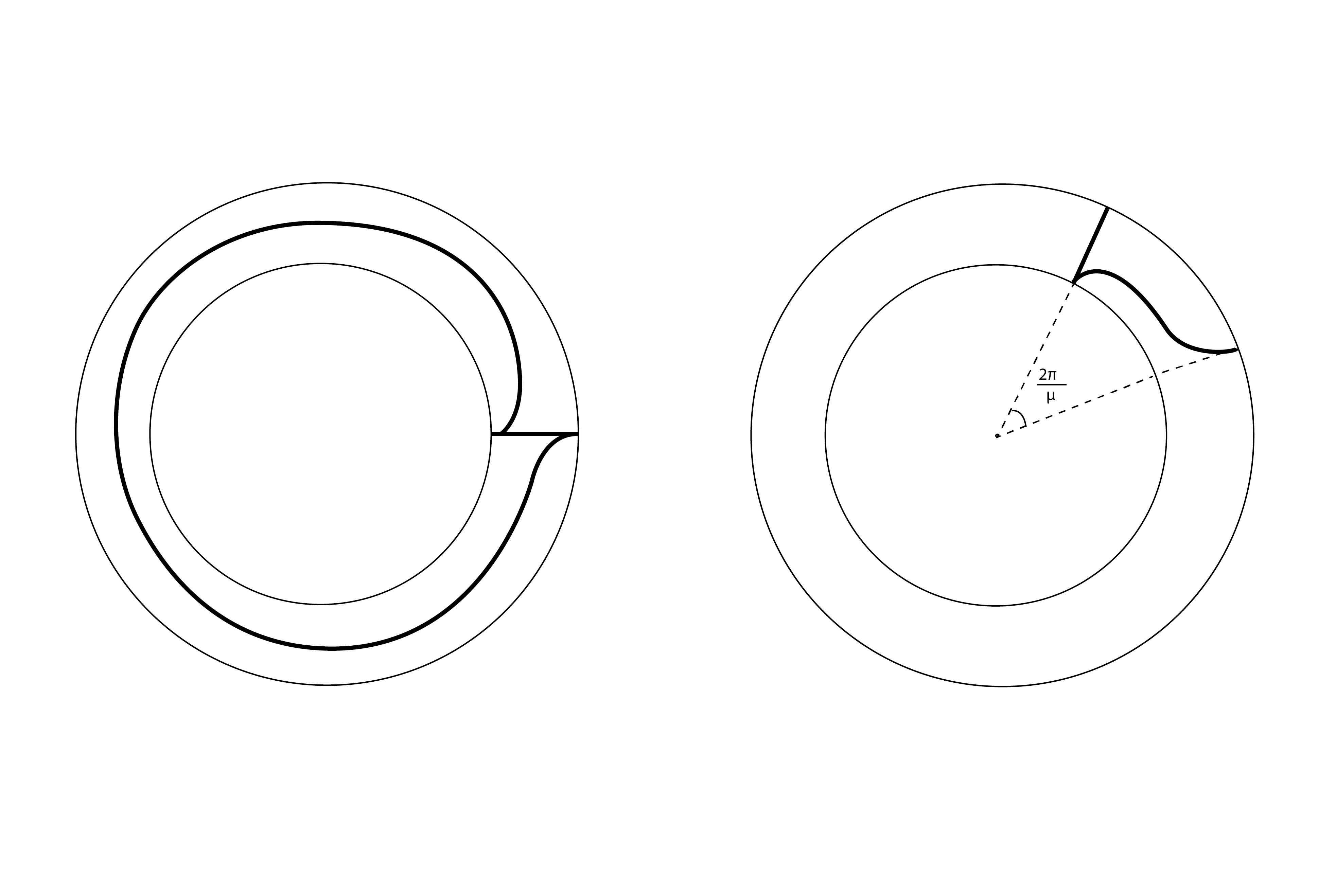}
\caption{On the left is the diffeomorphism of the base that gives the monodromy and on the right the one that gives $\rho_{a,\epsilon}$}
\label{monodromy}
\end{figure}

The following choice of perturbation and distinguished basis of vanishing paths for $p_a$ was introduced and analyzed at the Grothendieck group level in \cite{Seifert}. We consider the Morsification $$x_1+p_a(x_1,\ldots ,x_n).$$
We consider the diffeomorphism $\varphi:=\varphi_{1}\circ \mathrm{rot}_{2\pi/\mu(a)}$, where $\varphi_1$
was defined in the previous section and $\mathrm{rot}_{2\pi/\mu(a)}$ is the rotation by $\frac{2\pi}{\mu(a)}$
counterclockwise. Note that $\varphi$ preserves the set of critical values and
has a symplectomorphism lift $\Phi:=\rho_{a,1}\circ \psi_a$. Note also that $\varphi$ acts as identity outside
the outer boundary of the annulus on Figure \ref{monodromy}.

We choose a critical value, a positive real regular value that is outside of the support of $\varphi$ and a vanishing path $\gamma$ between them which lies outside of the circle containing the critical values. We choose $$\gamma,\varphi(\gamma),\ldots ,\varphi^{\mu(a)-1}(\gamma)$$ as our distinguished basis of vanishing paths.  We also grade the corresponding Lefschetz thimbles in a way that is compatible with a fixed graded lift of $\Phi$ (see Propositions \ref{prop-graded-lift} and \ref{prop-pert-symplecto}.)

\begin{rmk}\label{rem-A-grade}
We mentioned this convenient grading choice but we will not actually be using it. This is possible only because in this paper, on the $A$-side we are attempting to prove the Main Claim up to shifts. The grading convention that we just spelled out will without doubt play a role if one tries to upgrade the argument to a proof of the Main Claim taking into account the shifts. 
\end{rmk}

 This gives rise to a directed Fukaya-Seidel $A_\infty$-category $\mathcal{A}_a$ (which is a directed $A_\infty$-category) in the sense of \cite{Seidelmutation}.
We had temporarily called this category $F(p_a)$ above, but we will not do that anymore.

\begin{rmk}
We will make a definite choice of $\gamma$ in Section \ref{ss-A-outline} but note that because of the symmetry by the graded lift of $\Phi$ different choices give rise to equivalent directed $A_{\infty}$-categories.
\end{rmk}

For $a=\varnothing$ the empty tuple, we set $\mathcal{A}_{\varnothing}:=\CC_{\varnothing}$. 
This corresponds to the Fukaya-Seidel category obtained from the linear map $p_{\varnothing}:\mathbb{C}^0\to \mathbb{C}$ and a vanishing path. 

For $a=(a_1,\ldots,a_n)$ let us set
$$-a:=(a_2,\ldots,a_n), \ \ a-:=(a_1,\ldots,a_{n-1}).$$ 

\begin{conj}\label{conj-A-side-main}
$\mathcal{A}_a$ can be obtained from $\mathcal{A}_{-a}$ by the recursion $\mathcal{R}$. 
\end{conj}

As mentioned above, we give a detailed sketch of a computation strongly suggesting that this statement is true. We fall short of a full proof mainly because of some missing foundations in the theory of Fukaya-Seidel categories of tame Landau-Ginzburg models.

\begin{rmk}
Note that even though this statement is purely in terms of directed Fukaya-Seidel categories in their earliest incarnation from \cite{Seidelmutation}, our suggested proof crucially relies on the existence of a category which admits all thimbles as objects as is the case in the modern reincarnations. The main property whose proof is missing is the generation statement.
\end{rmk}

It is instructive to give a proof of this conjecture for tuples of length $1$. In this case our Morsification is 
$$x+x^{a_1}:\mathbb{C}\to \mathbb{C}.$$ 
Let us denote by $\mathcal{A}$ the directed $A_\infty$-category obtained from
the exceptional collection in the category of representations of the graded quiver $Q_{a_1-1}$ \[
\begin{tikzcd}
1 \arrow[r,"c"] & 2 \arrow[r,"c"] & \cdots \arrow[r,"c"] & a_1-1,
\end{tikzcd}
\]where $|c|=1$ and $c^2=0$,
given by the simple modules.
It is straightforward to show directly that $\mathcal{A}_{a_1}:= \mathcal{A}_{(a_1)}:$ is isomorphic to $\mathcal{A}$
but we will derive this from our general strategy.

First, we will show that $\mathcal{A}$ arises from $\CC_{\varnothing}$ via the recursion $\RR$ and then we will see how this is realized geometrically on the A-side.

The helix inside $Tw(\mathcal{C}_{\varnothing})$ is simply the only object $E$ of $\CC_{\varnothing}$
repeated over and over. Therefore, the directed $A_\infty$-category we obtain by keeping $a_1-1$ adjacent members of this list is the directed category with objects $P_1,\ldots,P_{a_1-1}$, where for every $i\leq j$, we have 
$Hom(P_i,P_j)=\mathbb{Z}[0]$ and all compositions are induced by multiplication in $\mathbb{Z}$. 
This can be identified with the exceptional collection given by the projective modules over the quiver $Q_{a_1-1}$.
Passing to right dual dual collection we obtain the collection given by the simple objects, so we get $\mathcal{A}$
as the result of the recursion.

\begin{figure}
\includegraphics[width=0.8\textwidth]{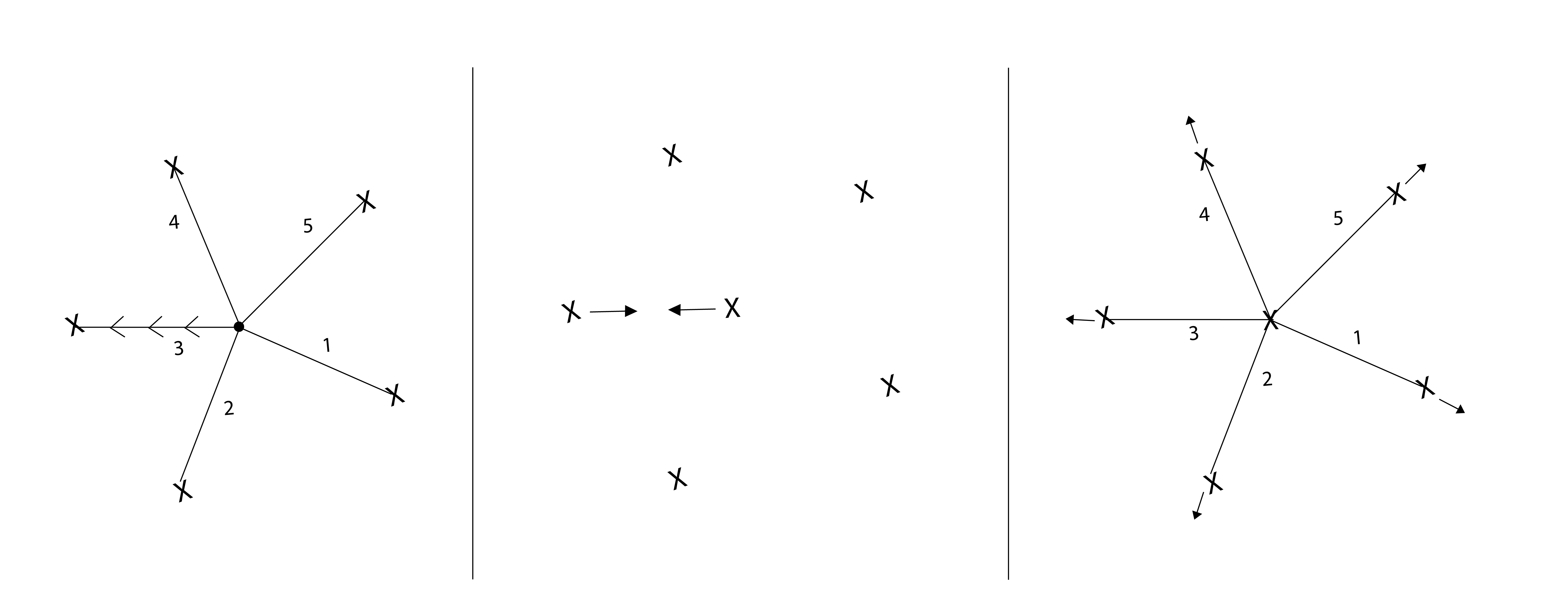}
\caption{On the left we see the critical values of $x+x^{a_1}$ along with the choice of vanishing paths that we use in the computation - they are obtained by dragging the reference fiber to the origin appropriately. The middle picture shows the matching path corresponding to the movement of the fiber shown on the left. The right picture shows all the matching paths and the behavior of outer critical values as $r\to 0$.}
\label{n=0}
\end{figure}

Geometrically, we are looking at the Lefschetz fibration $x+x^{a_1}:\mathbb{C}\to \mathbb{C}$ and a distinguished collection as described above. Corresponding to ``taking right dual exceptional collection" step, we compute the directed Fukaya-Seidel category associated to the left dual basis of vanishing paths. This can be computed inside the fiber over $0$ by appropriately moving (see Figure \ref{fig-move-A} for the same move with slightly different conventions) the base point - with the radial vanishing paths as shown in Figure \ref{n=0}. We consider the map $\{x+x^{a_1}=0\}\to \mathbb{C}$ given by projecting to the $x$ coordinate (i.e., the natural embedding). The vanishing cycles of the radial paths computed using the Lefschetz bifibration method are given by the radial matching paths, see Figure \ref{n=0}.
All the intersections (and structure maps) for the matching cycles are localized in the central fiber. In fact considering the family $x+rx^{a_1}$, where $r$ goes from $1$ to $0$ the directed categories with the continuously deformed matching paths do not change. When $r=0$ what we see is precisely the map $p_{\varnothing}:\mathbb{C}^0\to \mathbb{C}$ and radial paths going to infinity. The directed intersection numbers give the directed subcategory associated to the repetition of $E$ inside $Tw(\mathcal{C}_{\varnothing})$\footnote{Here we are omitting an explanation of how $\CC_{\varnothing}$ can be considered as a directed Fukaya-Seidel $A_{\infty}$ category of $p_{\varnothing}:\mathbb{C}^0\to \mathbb{C}$. This can be done using \cite{SeidelI}'s approach but it is confusing and not needed. To be able to interpret the distinct radial paths as objects of a geometric category the most natural option is to use a formalism similar to the one presented in Section \ref{ss-fukaya-seidel}.} that arose from the truncation of the helix in the previous paragraph. 

\begin{rmk}\label{rmk-extra-perturb} The strategy for $n>1$ is very similar but it involves an extra step. We would like to refer the reader to Remark 3.8 (and also Remark 3.7 for some related notation) of \cite{Seifert} for the immediate difficulty that arises when one applies the same strategy for $n>1$. 
What is achieved in the present paper relies on an additional perturbation (adding a small multiple of $x_2$) to the Morsification $x_1+p_a$ before we project the fiber above the origin to $\C$ using the $x_1$ coordinate (this projection is called $g_a$ in Remark 3.8 of \cite{Seifert} ). Note that we continue to use the fiber above $0$ and the radial paths as vanishing paths after the second perturbation. We also still analyze the fiber above $0$ by projecting it to the $x_1$ coordinate. The second perturbation breaks the $\mathbb{Z}/\mu(a)\mathbb{Z}$ symmetry and splits the fat singularity of the $x_1$ projection into $\mu(-a)$ non-degenerate singularities but it allows us to capture the information that was hidden in the very degenerate fiber above $0$ of the $x_1$-projection.
\end{rmk}

\subsection{More details on the B-side}

We consider the dg-category $\MF_{\Ga_a}(p_a)$ of $\Ga_a$-equivariant (or equivalently, $L_a$-graded) matrix factorizations of $p_a$.
For each $L_a$-homogeneous ideal $\II$ of the polynomial algebra $S$, 
there is a well-defined object of this category, which we denote by $\stab(\II)$: it is the
stabilization of the module $S/\II$, coming from
the relation between matrix factorizations and the singularity category (see e.g., \cite{Orlov-sing}).

Following Aramaki-Takahashi \cite{Aramaki} we consider the following graded matrix factorizations of $p_{a}$:
$$E:=\begin{cases} \stab(x_2,x_4,\ldots,x_n), & n \text{ even}\\
\stab(x_1,x_3,\ldots,x_n), & n \text{ odd}\end{cases}$$
The collection 
$$e_a:=(E,\tau(E),\ldots,\tau^{\mu^\vee(a)-1}(E))$$ is
a full exceptional collection in $\MF_{\Ga_{a}}(p_{a})$, where $\mu^\vee(a):=\mu(a^{\vee})=\mu(a_n,\ldots,a_1)$.
We refer to it as the {\it AT-collection}
and denote the corresponding directed $A_\infty$-category by $AT(a)$.

\begin{thm} (Theorem \ref{B-recursion-thm})
$AT(a)$ can be obtained from $AT(a-)$ by the recursion $\mathcal{R}$ up to shifts. 
\end{thm}

The first ingredient in the proof is a construction of 
a fully faithful functor 
\begin{equation}\label{B-side-fully-faithful-functor-eq}
\MF_{\Gamma_{a-}}(p_{a-})\to \MF_{\Gamma_{a}}(p_{a}).
\end{equation}
As was observed in \cite{FKK},
there is a natural such functor arising from the VGIT machinery of Ballard-Favero-Katzarkov \cite{BFK}.

The next step, based on explicit computations with matrix factorizations, is the identification of the image under the above functor
of the exceptional collection $e_{a-}$ with the left dual of the initial segment of the exceptional collection $e_a$.
This is done by a standard computation of morphisms between Koszul matrix factorizations.

The last step is the identification of
the directed $A_\infty$-algebra of $e_a$ with that of the part of the helix in the subcategory generated by the initial
segment, which we identified with $AT(a-)$. This is proved using some special features of the AT-collection. Namely, the key property is that for this collection we have
$\Hom^*(E,\tau^iE)=0$ for $i>\mu^\vee(a-)$ while the morphisms for the subcollection $(E,\tau E,\ldots, \tau^{\mu^\vee(a-)}E)$ form a Frobenius algebra (note that the length of this subcollection is one more
than the initial segment that corresponds to $AT(a-)$). Using this, plus a little bit more, we compute the image of  the left dual collection to the AT-collection under the
left adjoint functor to the inclusion \eqref{B-side-fully-faithful-functor-eq} and show that the corresponding directed $\Hom$-spaces are preserved. Strangely, our argument for this uses very little information about the functor
\eqref{B-side-fully-faithful-functor-eq}, but depends crucially on the properties of the $\Ext$-algebra of the 
Aramaki-Takahashi exceptional collection.


\subsection*{Structure of the paper}
Section \ref{s-A-side} is entirely about the A-side and contains our detailed strategy for the proof of Conjecture \ref{conj-A-side-main}. In Section \ref{ss-fukaya-seidel}, we give an overview of a Fukaya-Seidel category of thimbles. This section is rather conjectural and brief. In Section \ref{ss-A-outline}, we give an outline of our strategy and reduce the Main Claim to a concrete statement in Theorem \ref{Da-H-thm}. Section \ref{ss-roots} is an elementary section containing results about roots of a certain family of polynomials. These results then used to compute certain vanishing cycles as matching cycles in Section \ref{van-sph-sec}, which is the heart of the argument in the A-side.

Section \ref{s-B-side} is entirely about the B-side and contains our proof of Theorem \ref{B-recursion-thm}.
After recalling some basic tools from the theory of exceptional collections, we recall in Section \ref{AT-exc-coll-sec}
the definition and some properties of the Aramaki-Takahashi exceptional collection in the category of graded
matrix factorizations of chain polynomials. In Section \ref{VGIT} we outline the construction of the functor
\eqref{B-side-fully-faithful-functor-eq} and give a characterization of the image of the AT-collection under it.
In Section \ref{dual-exc-coll-sec}, we find a mutation functor that takes the image of the collection $AT(a-)$ under
\eqref{B-side-fully-faithful-functor-eq} to the dual collection to the initial segment of $AT(a)$.
Finally, in Section \ref{B-side-recursion-sec}, we combine the previous ingredients with some additional 
purely formal manipulations to prove Theorem \ref{B-recursion-thm}.

In Appendix A, we provide the simple Mathematica code used in discovering the statements of Section \ref{ss-roots} and Proposition \ref{prop-match-final}.


\subsection*{Acknowledgments}
U.V. thanks Paul Seidel for useful discussions and encouragement.
A.P. is grateful to Atsushi Takahashi for the helpful correspondence.
A.~P. is partially supported by the NSF grant DMS-2001224, 
and within the framework of the HSE University Basic Research Program and by the Russian Academic Excellence Project `5-100'.

\section{Computation on the A-side}\label{s-A-side}

Let us use the standard Fubini-Study Kahler structure on $\mathbb{C}^n$ along with the holomorphic volume form $\Omega=dz_1\wedge\ldots\wedge dz_n$ in what follows.
\subsection{A Fukaya-Seidel category of thimbles}\label{ss-fukaya-seidel}

Throughout this section let $f:\mathbb{C}^n\to \mathbb{C}$ be a tame Lefschetz (i.e. Morse) LG model in the sense of \cite{Fan}. Using Proposition 2.5 of \cite{Fan}, we see that for any $a\in\mathbb{Z}_{>1}^n$ and $\alpha_1,\ldots,\alpha_n\in\mathbb{C}$, $$p_a(x_1,\ldots,x_n)+\alpha_1x_1+\ldots+\alpha_nx_n:\mathbb{C}^n\to\mathbb{C}$$ is a tame LG model.

We will assume that the construction of the Fukaya-Seidel category introduced in an unpublished manuscript of Abouzaid-Seidel (see \cite{SeidelVI}: the $A_{\infty}$-category $\mathcal{A}$ as defined in equation (5.58) as a localization of the $A_{\infty}$-category $\mathcal{A}^{ord}$ that is defined in the first line of page 40) can be undertaken for $f$. We will call the resulting $A_{\infty}$-category $\mathcal{F}(f)$. Below we discuss some properties of
this $A_\infty$-category referring to \cite{SeidelVI} for details.

Let us call path $p$ in the base of $f$ a horizontal at infinity (HAI) vanishing path if it can be parametrized by a smooth proper embedding $\gamma:[0,\infty)\to \mathbb{C}$ satisfying the following properties
\begin{itemize}
\item $\gamma(t)$ is a critical value if and only if $t=0$.
\item For some $t_0>0$ and $ord(p)\in \mathbb{R}$, $Re(\gamma(t))>0$ and $Im(\gamma(t))=ord(p)$ for all $t\geq t_0$.
\end{itemize}
Let us call $ord(p)$ the ordinal of $p$.

To each HAI vanishing path, we can associate a (Lefschetz) thimble, which is an embedded non-compact Lagrangian submanifold of $\mathbb{C}^n$. By equipping these thimbles with gradings, we can view them as Lagrangian branes, which we call graded thimbles.

Let $L_1,\ldots , L_N$ be an ordered collection of graded Lagrangians in $\mathbb{C}^n$ each of which is either a closed exact Lagrangian sphere or a graded thimble of a HAI vanishing path. We also make the crucial assumption that the no two of the HAI vanishing paths have the same ordinal. Then, we can define a directed $A_{\infty}$-category $Fuk^{\to}(L_1,\ldots , L_N)$ with the ordered list of objects corresponding to $L_1,\ldots,L_N$ using 
\begin{itemize}
\item consistent choices of compactly supported Hamiltonian perturbations to make Lagrangians transverse (directedness really helps here); 
\item almost complex structures which agree with the standard complex structure of $\mathbb{C}^n$ outside of a compact subset
\end{itemize}
to define the structure maps. $Fuk^{\to}(L_1,\ldots , L_N)$ is well defined up to $A_{\infty}$-quasi-isomorphism respecting the ordering of the objects. This is standard (see \cite{SeidelI} for example) except obtaining the necessary $C^0$-estimates in our particular set-up.

Let us give more details on one of the few possible approaches on obtaining the $C^0$ bounds. A standard application of the open mapping principle shows that 
all of the curves that are solutions of the various perturbed pseudo-holomorphic curve equations that we need to consider in the procedure project into a compact subset $K\subset \mathbb{C}$ of the base of $f$. To deal with escaping to infinity within $f^{-1}(K)$ we can use monotonicity techniques since $L_i\cap f^{-1}(K)$ is compact for all $i=1,\ldots ,N$ and the standard flat metric on $\mathbb{C}^n$ is geometrically bounded.

Let us now recall very briefly what the objects of  $\mathcal{F}(f)$ are in the Abouzaid-Seidel approach. For every homotopy class of HAI vanishing paths let us choose a representative path $p_0$.
Next, for each graded thimble $T(p_0)$ over $p_0$,
we choose an infinite sequence $T(p_1),T(p_2),\ldots$ of graded thimbles,
such that the underlying HAI vanishing paths $p_i$
are homotopic to $p_0$ and the gradings are transported from $T(p_0)$, and such that
the sequence of real numbers $ord(p_i)$ is strictly increasing and tends
to infinity. Objects of $\mathcal{F}(f)$ are all the graded thimbles obtained as a result of this procedure (we assume that
our choices of paths are sufficiently generic).
Note that the objects $T(p_i)$ are all quasi-isomorphic to $T(p_0)$ as objects of $\mathcal{F}(f)$.

\begin{rmk}\label{rmk-C0-bounds} To achieve this last crucial point, Abouzaid-Seidel procedure involves localizing an auxilary $A_{\infty}$-category at certain continuation elements. Obtaining the $C^0$ estimates that are necessary to define these elements and prove that they satisfy the desired properties is non-trivial. The relevant perturbed pseudo-holomorphic curve equations involve moving boundary conditions (thimbles moving at infinity), which makes it difficult to use the open mapping principle. Therefore one needs to rely entirely on monotonicity techniques. Even though we fully believe that this can be done, we do not explain how to do it. This is one of the remaining steps to turn our strategy into a full proof.\end{rmk}


Given HAI vanishing paths $p_1,\ldots ,p_N$ and graded thimbles $T_1,\ldots T_N$ above them, which are assumed to be objects of $\mathcal{F}(f)$, 
one has a concrete way of computing the directed $A_{\infty}$-subcategory of the ordered collection $T_1,\ldots T_N$ in $\mathcal{F}(f)$. Namely, we find graded thimbles $\tilde{T}_1,\ldots ,\tilde{T}_N$ (not necessarily objects of the category) such that HAI vanishing paths $\tilde{p}_1,\ldots ,\tilde{p}_N$ are in the same homotopy class with $p_1,\ldots ,p_N$, respectively, and the brane structure on $\tilde{T}_i$ is transported from $T_i$, with the following property\begin{itemize}
\item the ordinals of $\tilde{p}_1,\ldots ,\tilde{p}_N$ are strictly decreasing.

\end{itemize}
Then the directed $A_{\infty}$-category $Fuk^{\to}(\tilde{T}_1,\ldots ,\tilde{T}_N)$ is quasi-isomorphic to the directed $A_{\infty}$-subcategory we are interested in, where $\tilde{T}_i$ is sent to $T_i.$ We call this the Computability property of $\mathcal{F}$. See Figure \ref{computedirect} for a depiction of the process.

\begin{figure}
\includegraphics[width=0.5\textwidth]{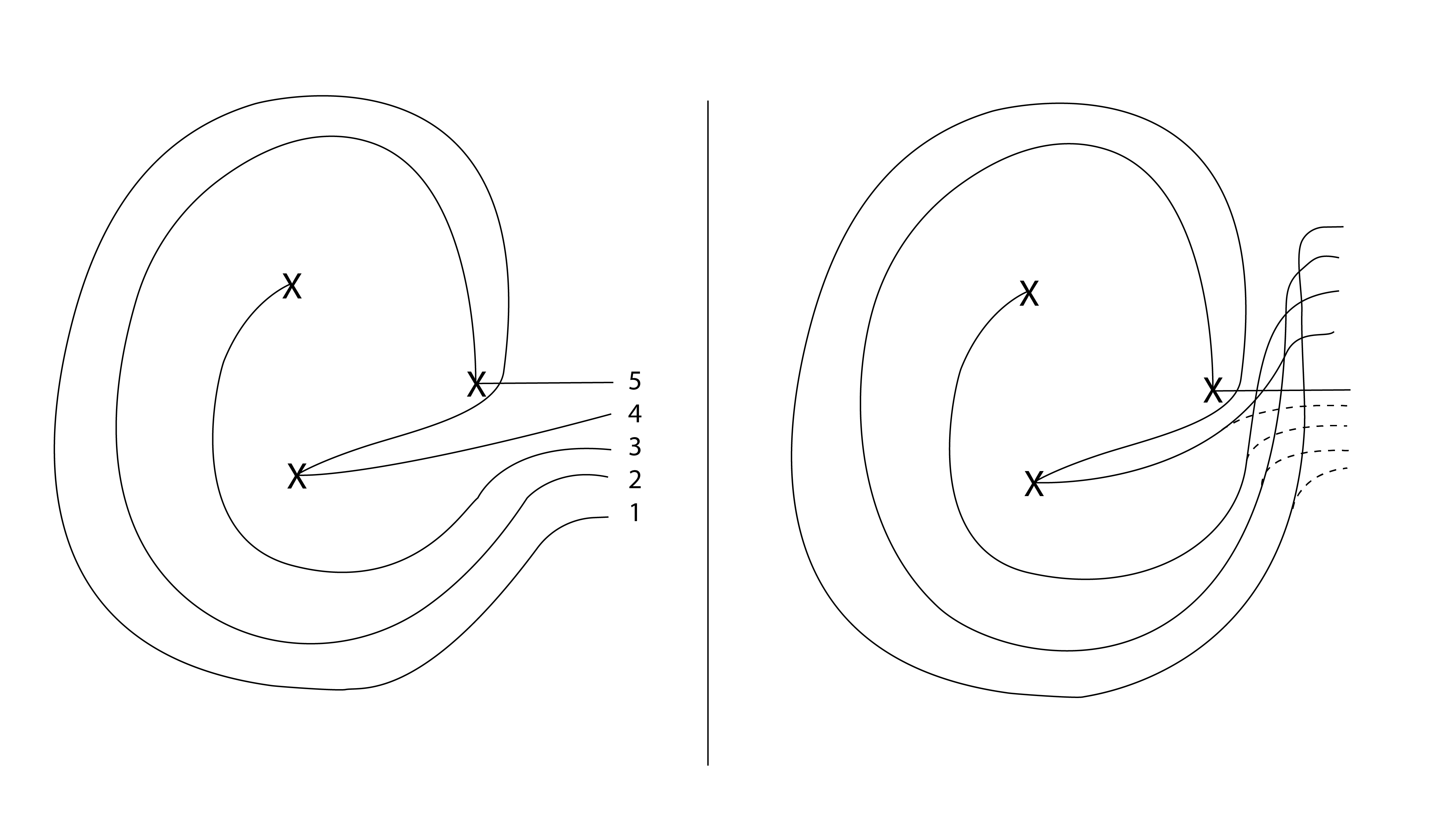}

\caption{On the left we see HAI vanishing paths of objects ordered as written. Their directed subcategory can be computed using the HAI vanishing paths on the right.}\label{computedirect}
\end{figure}

\begin{rmk}
For the $A_{\infty}$-category constructed in Seidel's book (Section 18 of \cite{Seidelbook}) such a computation involves the double covering trick and computing the invariant part of a certain $A_{\infty}$ algebra of closed Lagrangians under a $\mathbb{Z}/2$ action. This makes it hard to use in our argument. 
\end{rmk}

There is a more refined version of the Computability property if $p_1,\ldots ,p_N$ are pairwise disjoint paths with $ord(p_1)<\ldots<ord(p_n)$. We choose a sufficiently large positive integer $A$ and bend the paths to $\tilde{p}_1,\ldots ,\tilde{p}_N$ near $Re=\infty$ such that they all pass through $(A,0)$, but do not intersect otherwise. Then, we obtain an ordered collection of graded Lagrangian spheres (vanishing cycles) $V_1,\ldots ,V_N$ inside $f^{-1}(A)$. Now, we can define a directed Fukaya-Seidel category $FS^{\rightarrow}(V_1,\ldots ,V_N)$  as in \cite{Seidelmutation}. Combining the results of \cite{SeidelI} with the Computability property, we can show that $FS^{\rightarrow}(V_1,\ldots ,V_N)$ is quasi-isomorphic to the subcategory of $T_1,\ldots T_N$ with $V_i$ mapping to $T_i.$ Let us call this the Computability in the fiber property.

The usefulness of $\mathcal{F}(f)$ is entirely due to the following generation property. We first state it and then briefly explain the terms used in it.

\begin{conj} (Generation by distinguished collections)
 Yoneda images of a sequence of objects of $\mathcal{F}(f)$ which correspond to a distinguished collection of graded thimbles generate $Tw(\mathcal{F}(f))$.
\end{conj}

\begin{rmk}\label{rmk-generation}
It is widely expected that this property will follow from a geometric translation to the Weinstein sector framework, but this has not been done in the literature yet. This is the main missing piece from our strategy being a full proof.
\end{rmk}

A collection of pairwise collection HAI vanishing paths, one for each critical value, is called a distinguished collection of HAI vanishing paths. Choosing an arbitrary brane structure on each of the thimbles gives what we called above a distinguished collection of graded thimbles. Note that such a collection $T_1,\ldots,T_n$ can be naturally ordered by requiring that the corresponding paths $p_1,\ldots,p_n$ satisfy $ord(p_1)<\ldots<ord(p_n)$.
With this order $(T_1,\ldots,T_n)$ is an exceptional collection in $Tw(\mathcal{F}(f))$, and the above conjecture
states that this exceptional collection is full.

The following weak version of the old conjecture ``monodromy gives a Serre functor" is crucial in our argument. Its proof is quite simple given the Generation by distinguished collections property.

\begin{prop} (Geometric helix equals algebraic helix)
Consider a collection of homotopy classes of HAI vanishing paths $\{\gamma_i\}_{i\in \mathbb{Z}}$ such that \begin{itemize}
\item $\gamma_1,\ldots \gamma_n$ can be represented by a distinguished collection of HAI vanishing paths
\item For every $i\in \mathbb{Z}$, a representative of $\gamma_{i-n}$ is given by applying the monodromy diffeomorphism (see the left side of Figure \ref{monodromy}) to a representative of $\gamma_i$.
\end{itemize}
Assume that $\{T_i\}_{i\in \mathbb{Z}}$ are some corresponding objects of $\mathcal{F}(f)$. The brane structures can be chosen such that the Yoneda images of this collection forms a helix inside $Tw(\mathcal{F}(f))$.
\end{prop}

\begin{proof}[Proof sketch] 
From Figure 3 (which gives an example with $n=3$) we see that $\Hom(T_i,T_0)=0$ for $i=1,\ldots,n-1$, and $\Hom(T_n,T_0)$ is $1$-dimensional.
This implies that $T_0$ with an appropriate brane structure is the left mutation of $T_n$ through $\lan T_1,\ldots,T_{n-1}\ran$.
Similarly, $\Hom(T_{n+1},T_i)=0$ for $i=2,\ldots,n$, and $\Hom(T_{n+1},T_1)$ is $1$-dimensional. Hence, $T_{n+1}$ with an appropriate brane structure is the right mutation of
$T_1$ through $\lan T_2,\ldots,T_n\ran$. Since the helix is obtained by iterating these two kinds of mutations, our assertion follows.
\end{proof}

We will also use the following geometric realization of dual exceptional collections 
(see Sec.\ \ref{exc-coll-basics-sec} for the definitions concerning exceptional collections). The proof is again straightforward assuming generation by distinguished collections.

Given a homotopy class of a distinguished collection of HAI vanishing paths $[\{\gamma_i\}_{i=0}^n]$, we can talk about the left and right dual homotopy class of a distinguished collection of HAI vanishing paths. The left (resp. right) dual admits a representative distinguished collection $\{^\vee\gamma_i\}_{i=-n}^0$ (resp. $\{\gamma_i^\vee\}_{i=n}^{2n}$)  all of whose ordinals are smaller (resp. larger) than the ordinals of $\gamma_i$, $i=0,\ldots ,n$ and $\gamma_i$ and $^\vee\gamma_j$ (resp.$\gamma_k^\vee$) can only intersect at a critical value for all $i=0,\ldots ,n$ and $j=-n,\ldots ,0$ (resp. $k=n,\ldots ,2n$).

\begin{prop} (Geometric dual equals algebraic dual)
Consider a homotopy class of a distinguished collection of HAI vanishing paths $[\{\gamma_i\}_{i=0}^n]$ and let $[\{^\vee\gamma_i\}_{i=-n}^0]$ be the left dual.
Assume that $\{T_i\}_{i=0}^n$ and $\{T_i\}_{i=-n}^0$ are some corresponding objects of $\mathcal{F}(f)$. Up to shifts, the Yoneda images of the $\{T_i\}_{i=-n}^0$ give the left dual exceptional collection to the one of $\{T_i\}_{i=0}^n$ inside $Tw(\mathcal{F}(f))$. The analogous statement holds for the right duals.
\end{prop}

\subsection{Outline of the recursion on the A-side}\label{ss-A-outline}

For an $n$-tuple of positive integers $a=(a_1,..,a_n)\in \mathbb{Z}^n_{>1}$, $n\geq 1$, we define the polynomial:
\begin{align}p_a(z_1,\ldots,z_n):= -z_1^{a_1}z_{2}+z_2^{a_2}z_{3}-\ldots +(-1)^{n}z_n^{a_n}.
\end{align}

Note that we have changed the signs of some of terms from the original definition of $p_a$ given in the introduction . This choice makes the critical point computations much cleaner.  It is straightforward to relate our statements here to the statements in the introduction by simple diagonal changes of variables.

Let us also define $g_a:\mathbb{C}^n\to \mathbb{C}$ as the Lefschetz fibration given by $$z\mapsto z_1+p_a(z).$$ 
Recall that we defined $$\mu(a)=\mu(a_1,\ldots,a_n)=a_1\ldots a_n-a_2\ldots a_n+a_3\ldots a_n-\ldots$$ in the introduction. It is well known that $\mu(a)$ is the  Milnor number of the singularity of $p_a$. For a discussion of the convenient numerics of $\mu(a)$ see Section \ref{AT-basic-sec}. The map $g_a$ has $\mu({{a}})$ critical points, and the corresponding critical values are distinct and placed equiangularly on a circle centered at the origin. One of the critical values is on the positive real axis. For proofs of these statements see Appendix A in \cite{Seifert}.
Furthermore, the fact that the number of critical points of $\epsilon z_1+p_a(z)$ for all $\epsilon\in\mathbb{C}^*$ is equal to the Milnor number $\mu(a)$ implies that the
critical points of $g_a$ are nondegenerate.

Let us fix a large positive real number $A$ and introduce some vanishing paths in the base of $g_a$ whose one end is at $A$ and none of which intersect the positive real axis to the right of $A$. Figure \ref{fspiral} should help the reader follow along. We will call some of our vanishing paths standard and others dual. We will not be careful about distinguishing between vanishing paths and their homotopy classes.


We first describe the standard vanishing paths $\{\gamma_i\}_{i=-\infty}^{\infty}$. These are indexed by integers and the one corresponding to $0$, i.e.
$\gamma_0$, is the straight path from the positive real critical value to $A$.  Recall that we defined the diffeomorphism $\varphi:\mathbb{C}\to \mathbb{C}$ as the composition $\varphi:=\varphi_{1}\circ \mathrm{rot}_{2\pi/\mu(a)}$ in Section \ref{ss-Aside-intro}. Note that $\varphi$ preserves the set of critical values and
has a symplectomorphism lift $\Phi:=\rho_{a,1}\circ \psi_a$. For all $i\in \mathbb{Z}$, we define $$\gamma_i:=\varphi^{i}(\gamma_0).$$

Second, we introduce the dual vanishing paths $\{^\vee\gamma_i\}_{i=-\mu(a)+1}^{0}$ as the left dual distinguished collection of vanishing paths to the distinguished collection $\{\gamma_i\}_{i=0}^{\mu(a)-1}$. These are the dashed paths from Figure \ref{fspiral}.

\begin{figure}
\includegraphics[width=0.7\textwidth]{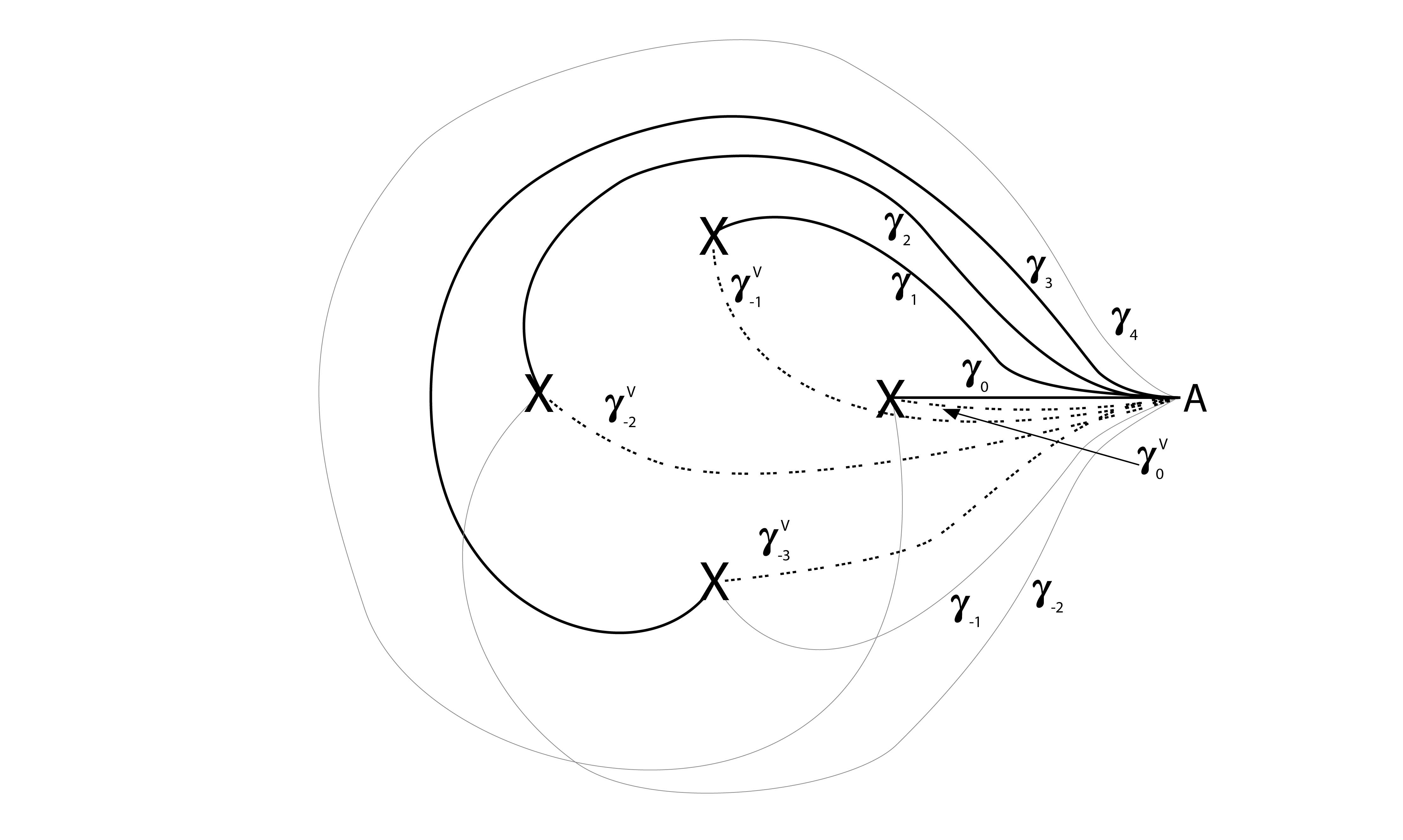}
\caption{Vanishing paths in the base of $g_a$}
\label{fspiral}
\end{figure}

In what follows we will not be keeping track of the gradings of Lagrangian branes, and only talk about the underlying Lagrangian submanifolds, see Remark \ref{rem-A-grade}. This is of course an abuse, but we believe it will not cause confusion. Since, we will not be able to keep track of the gradings in our arguments, adding grading data would only result in cluttering up the notation.

Let $(\mathcal{A}_a; E_0,\ldots,E_{\mu(a)-1})$ be the directed Fukaya-Seidel $A_{\infty}$ category with the exceptional collection defined using the vanishing Lagrangian spheres of $\gamma_0,\ldots,\gamma_{\mu(a)-1}$ as in \cite{Seidelmutation}.
\begin{rmk}Note that because of the symmetry by $\Phi$ 
the directed $A_{\infty}$-categories defined using $\gamma_k,\ldots,\gamma_{k+\mu(a)-1}$ are quasi-isomorphic for all $k\in\mathbb{Z}$ where the ordering of the objects is preserved. \end{rmk}

Let us call $\mathcal{D}_a$ the directed Fukaya-Seidel $A_{\infty}$ category of the Lagrangian vanishing spheres of $^\vee\gamma_{-\mu(a)+1},\ldots,^\vee\gamma_{0}.$ The following proposition can be proven using the results in \cite{Seidelbook}. 
Note that it can also be deduced formally
from the properties of $\mathcal{F}(g_a)$ discussed in Section \ref{ss-fukaya-seidel} (Generation by a distinguished collection, Computability in the fiber and Geometric dual equals algebraic dual).

\begin{prop}\label{Da-Aa-prop}
$\mathcal{D}_a$ is quasi-isomorphic to the $A_{\infty}$ subcategory of $Tw(\mathcal{A}_a)$ corresponding to the exceptional collection left dual to the Yoneda image of the defining exceptional collection of $\mathcal{A}_a.$
\end{prop}

Finally, let $N$ be a positive integer and let us consider an ordered collection of HAI vanishing paths in the base of $g_a$ defined as follows. Let $\tilde{\gamma}$ be the HAI vanishing path starting at the positive real critical value and going along the real axis. Consider the collection 
$$\varphi^{-N+1}(\tilde{\gamma}),\ldots ,\varphi^{-1}(\tilde{\gamma}),\tilde{\gamma}$$ and isotope them slightly (keeping them HAI vanishing paths) to 
\begin{equation}\label{tw-gamma-collection}
\tilde{\gamma}_{-N+1},\ldots ,\tilde{\gamma}_{-1},\tilde{\gamma}_0
\end{equation}
so that the ordinals of $\tilde{\gamma}_{-N+1},\ldots ,\tilde{\gamma}_{-1},\tilde{\gamma}_0$ are strictly decreasing. As we discussed in Section \ref{ss-fukaya-seidel}, we can define a directed $A_{\infty}$-category of the graded thimbles of $\tilde{\gamma}_{-N+1},\ldots ,\tilde{\gamma}_1,\tilde{\gamma}_0$. Let us call this category $\mathcal{H}_a(N).$

The following Proposition follows from the Computability,  Computability in the fiber, Generation by a distinguished collection and Geometric helix equals algebraic helix properties of $\mathcal{F}(g_a)$ as discussed in Section \ref{ss-fukaya-seidel}. We are not aware of a proof that only relies on results in existing literature. 

\begin{prop}\label{Ha-prop}
$\mathcal{H}_a(N)$ is quasi-isomorphic to the directed $A_{\infty}$ subcategory of $Tw(\mathcal{A}_a)$ corresponding to the length $N$ truncation of the helix generated by  
(Yoneda image of) the exceptional collection $E_0,\ldots,E_{\mu(a)-1}$ with the last element of the truncated helix being $E_0$.
\end{prop}

\begin{proof}
Let us fix arbitrary objects $\{T_i\}_{i\in \mathbb{Z}}$ of $F(g_a)$ corresponding to the collection $\{[\varphi^i(\tilde{\gamma})]\}_{i\in \mathbb{Z}}$ of homotopy classes of HAI vanishing paths. 
By the generation and computability in the fiber properties, we have a quasi-equivalence of triangulated $A_{\infty}$-categories $$Tw(\mathcal{A}_a)\to Tw(F(g_a))$$
sending $E_0,\ldots,E_{\mu(a)-1}$ to $T_0,\ldots,T_{\mu(a)-1}$. 

It suffices to prove that $\mathcal{H}_a(N)$ is quasi-isomorphic to the directed $A_{\infty}$ subcategory of $Tw(F(g_a))$ corresponding to the truncated helix of length $N$ of the Yoneda images of $T_0,T_1,\ldots, T_{\mu(a)-1}$ with the last element of the truncated helix being $T_0$.

We now use the Geometric helix equals algebraic helix property for the collection $\{[\varphi^i(\tilde{\gamma})]\}_{i\in \mathbb{Z}}$ of homotopy classes of HAI vanishing paths and objects $\{T_i\}_{i\in \mathbb{Z}}$. Note that $\varphi^{-\mu(a)}$ is a monodromy diffeomorphism. 
As a result, the (Yoneda images of) $(T_i)_{i\in\Z}$ is a helix generated by $T_0,\ldots,T_{\mu(a)-1}$. 

We should take the truncation $(T_{-N+1},\ldots,T_0)$ of the helix $(T_i)$ and match it with the category $\HH_a(N$).
For this we observe that the Computability property gives a quasi-isomorphism of the directed category with the objects $T_{-N+1},\ldots ,T_1,T_0$ and $\HH_a(N)$ (since
by construction the ordinals of $\tilde{\gamma}_{-N+1},\ldots ,\tilde{\gamma}_1,\tilde{\gamma}_0$ are strictly decreasing).
\end{proof}

We apply Proposition \ref{Ha-prop} with $-a$ instead of $a$ and with $N=\mu(a)$. This will give a geometric realization of the first part of the recursion, namely of
the category generated by the truncated helix of length $\mu(a)$ in $Tw(\mathcal{A}_{-a})$. Since the left dual to the natural collection in $Tw(\mathcal{A}_a)$
is realized geometrically in Proposition \ref{Da-Aa-prop},
we will know that the category $\mathcal{A}_a$ is obtained from $\mathcal{A}_{-a}$ by recursion $\mathcal{R}$, once we prove the following statement.

\begin{thm}\label{Da-H-thm}
There is an equivalence up to shifts
$$\mathcal{D}_a\to\mathcal{H}_{-a}(\mu(a))),$$
of directed $A_{\infty}$-categories.
\end{thm}

We will prove Theorem \ref{Da-H-thm} in Section \ref{van-sph-sec}.



\subsection{Roots of a family of polynomials}\label{ss-roots}

The results of this section will be used in computing certain matching paths in the next section. 

Let $t,s$ be complex numbers and $c$ a positive real number. We consider the following equation in $\mathbb{P}^1$: 
\begin{equation}\label{proj-crit-pol-eq}
y^{\mu}x^{\mu_-}=c(sy^a+tx^a)^{d_{-}},
\end{equation}
where $(y:x)$ are the homogeneous coordinates, and
$\mu,\mu_-,a,d_-$ are positive integers satisfying $$\mu+\mu_-=d:=ad_-.$$
We will be interested in how the roots of this equation vary when we vary $c,t,s$ in a certain region.

Fix $c$. Note that for $s=t=0$, we have one root with multiplicity $\mu_-$ at the point $x=0$ (called $0$) and another one with multiplicity $\mu$ at $y=0$ (called $\infty$). Once we make $s$ non-zero, the root at $0$ splits into $\mu_-$ simple roots.
We are going to keep $|s|$ sufficiently small (with some bound depending on $c,\mu,\mu_-,a,d_-$) and positive but arbitrary otherwise. Then, we will show that turning on the $t$ parameter does not change the locations of the $\mu_-$ simple roots 
near $0$
``too much" unless $|t|$ becomes larger than a number depending only on $c$, most importantly independently of $s$. In particular, it is possible for $|t|$ to be much larger than $|s|$ in this statement. We will specify what ``too much" means below - indeed we have something specific in mind. As a first approximation to why something like this might true let us note that if we keep $s=0$, then no matter how large $|t|$ is, the multiplicity $\mu_-$ root at $0$ never moves. If the reader has access to Mathematica, we provided a simple code in the Appendix to experiment with the roots of this family of polynomials.

Let $\mathbb{A}_1:=\mathbb{A}_{\frac{x}{y}}$ and $\mathbb{A}_2:=\mathbb{A}_{\frac{y}{x}}$ be the standard affine charts in $\mathbb{P}^1$. Let us equip them with the standard Kahler structure for their chosen affine coordinate.

Let us set $z=\frac{x}{y}$. The equation in $\mathbb{A}_1$ becomes 
\begin{align}\label{critpol} z^{\mu_-}=c(s+tz^a)^{d_-}.\end{align} Below we will analyze the roots of this equation but all results hold equally well in the other chart
(with the roles of $t$ and $s$ swapped). We also assume that $c=1$, noting that the general case can be recovered by rewriting $t$ and $s$ as $c^{1/d_-}t$ and $c^{1/d_-}s$. 

For $\gamma\in [0,2\pi)$, let $R_{\gamma}$ denote the ray in the complex plane starting from the origin that makes a positive angle of $\gamma$ with the positive real axis. For any $\psi\in (0,2\pi)$, let $N_{\psi}(R_{\gamma})$ be the conical region in the plane consisting of points (seen as vectors starting at the origin) that make less than $\frac{\psi}{2}$ angle with $R_{\gamma}$ (in positive or negative directions). 

For every $\epsilon>0,$ $n$ positive integer, $\phi >0$ such that $2n\phi<2\pi$, and $\gamma\in [0,2\pi)$ we define $$Dart(\epsilon,n,\phi,\gamma):=\{(z\in\mathbb{C}\mid |z|<\epsilon \text{ and }z^n\in N_{2n\phi}(R_{\gamma})\}.$$

\begin{prop}\label{small-roots-prop} 
Let us divide the solutions of the Equation  \eqref{critpol} with $c=1$ into two groups: the ones that lie inside the closed disk of radius $\frac{1}{2}$ in $\mathbb{A}_1$ (small roots) and the others (large roots). There exists a
positive constant $C=C(a,\mu_-,d_-)$ depending only on $a,\mu_-,d_-$ with the following properties.
\begin{enumerate}
\item For all $|t|\leq 1$ and $0<|s|<C$, there are $\mu_-$ many small roots.
\item For $|t|\leq 1$ and $0<|s|<C$, there exist $0<\epsilon(s)=\epsilon(|s|)<\frac{1}{2}$, $0<\phi(s)<\frac{\pi}{\mu_-}$, and $\gamma(s)\in [0,2\pi)$ with the following properties:
\begin{itemize}
\item There is exactly one small root inside each connected component of $$Dart(\epsilon(s),\mu_-,\phi(s),\gamma(s))\subset \mathbb{A}_1.$$ 
\item As $s\to 0$, $\epsilon(s)$ and $\phi(s)$ converge to $0$.
\item $\gamma(s)$ is the argument of $s^{d_-}$ valued in $[0,2\pi).$ 
\item All small roots are simple. 
\end{itemize}
\end{enumerate}
\end{prop}

\begin{proof}We follow the strategy of the proof of Theorem 4.1 in Melman's beautiful paper \cite{Melman}. In particular, his Lemma 2.7 will play a very crucial role.

We rewrite Equation \eqref{critpol} with $c=1$ as
 \begin{align}\label{recritpoly}
 (z^{\mu_-}-s^{d_-})-(d_-s^{d_--1}tz^a+\ldots+t^{d_-}z^d)=0.
 \end{align}

Let us prove (1). We will use Rouche's theorem (e.g. Theorem 2.1 in \cite{Melman}). For $|z|= 1/2$, $|t|\leq 1$ and 
$|s|\leq 1$, we have the following two inequalities:
\begin{align*}
|z^{\mu_-}-s^{d_-}|\geq (1/2)^{\mu_-}-|s|^{d_-}
\end{align*} 
\begin{align*}
|d_-s^{d_--1}tz^a+\ldots+t^{d_-}z^d|\leq |d_-s^{d_--1}tz^a|+\ldots+|t^{d_-}z^d| < |s|C+ (1/2)^{d},
\end{align*}where $C$ is a constant depending on $a$ and $d_-$. Hence, using $\mu_-<d=ad_-$, for sufficiently small $|s|$, we have 
\begin{align}\label{rouche}
|z^{\mu_-}-s^{d_-}|> |d_-s^{d_--1}tz^a+\ldots+t^{d_-}z^d|. 
\end{align} 
Therefore, the number of solutions of the Equation \eqref{recritpoly} inside the disk of radius $1/2$ centered at the origin is the same as the number of solutions of $z^{\mu_-}=s^{d_-} $ in the same region, as desired.

Now let us proceed to prove (2). This is again an application of Rouche's theorem. Let $|t|\leq 1$, and $|s|<1$ be sufficiently small as required by the previous step. Moreover, $|s|$ should also satisfy a possibly stronger bound that we will explain now. Using again that $\mu_-<d$, we can choose $\delta>0$ such that $\de<d/\mu_- -1$.
Now we require that $|s|$ satisfies the inequality
$$|s|^{d_-+\delta}> d_-|s|^{d_--1}(|2s|^{\frac{d_-}{\mu_-}})^a+\ldots+d_-|s|(|2s|^{\frac{d_-}{\mu_-}})^{d-a}+(|2s|^{\frac{d_-}{\mu_-}})^d.$$The right hand side of this inequality is obtained by inputting $1$ for each $t$, $|s|$ for $s$ and $|2s|^{\frac{d_-}{\mu_-}}$ for $z$ in the expression $d_-s^{d_--1}tz^a+\ldots+t^{d_-}z^d$ as in Equation \eqref{recritpoly}. To see that for sufficiently small $|s|$ this inequality is satisfied note that the power of $|s|$ in each term of the RHS is strictly bigger than $d_-+\delta$:
$$d_--k+\frac{d_-}{\mu_-}ka\ge d_- +\frac{d_-a}{\mu_-}-1>d_-+\delta,$$ 
for $k=1,\ldots , d_-$. 

Let us define $$\epsilon:=|s|^{d_-+\delta},$$and note that $\epsilon <|s|^{d_-}$.
Note that if $|z^{\mu_-}-s^{d_-}|\le\epsilon$, then $|z|^{\mu_-}<2|s^{d_-}|$, and therefore $$|z|<|2s|^{\frac{d_-}{\mu_-}}.$$

This time we will apply Rouche's theorem in the connected components of the domain in $z$ described by the inequality \begin{align*}
|z^{\mu_-}-s^{d_-}|\leq \epsilon.
\end{align*} For $s\neq 0$ this domain has $\mu_-$ simply connected components all of which are contained in the disk of radius $1/2$ centered at the origin (assuming $s$ is small).

We now again consider Equation \eqref{recritpoly}. We want to prove that the Inequality \eqref{rouche} holds on the set $|z^{\mu_-}-s^{d_-}|=\epsilon$. This follows immediately since $$\epsilon>d_-|s|^{d_--1}(|2s|^{\frac{d_-}{\mu_-}})^a+\ldots+d_-|s|(|2s|^{\frac{d_-}{\mu_-}})^{d-a}+(|2s|^{\frac{d_-}{\mu_-}})^d>|d_-s^{d_--1}tz^a+\ldots+t^{d_-}z^d|.$$

Hence we obtain that each connected component of $\{|z^{\mu_-}-s^{d_-}|\leq \epsilon\}$ contains exactly one solution. These are all the small roots. To relate these regions to the dart-like regions in the statement we use Lemma 2.7 of \cite{Melman}. All four bullet points follow.
\end{proof}

To state the following corollary which is what we will directly use in later chapters, we make a new definition. For every $r>0,$ $n$ positive integer, $\phi >0$ such that $2n\phi<2\pi$,
$$Dart^{\infty}(r,n,\phi):=\{(z\in\mathbb{C}\mid |z|>r\text{ and }z^n\in N_{2n\phi}(R_{0})\}.$$

\begin{cor}\label{cor-dart} Let $t$ be a real number and $s$ a complex one. Let us call the solutions of the Equation  \eqref{critpol} that lie inside the closed disk of radius $\frac{1}{2}$ the small roots and the ones that lie outside the closed disk of radius $2$ the large roots. 

Then, there exists a
positive constant $C=C(a,\mu_-,d_-,c)$ depending only on $a,\mu_-,d_-,c$ such that for all $0<t<C$ and $0<|s|<C$.
\begin{enumerate}
\item There are $\mu_-$ many small roots and $\mu$ large roots. In particular, all roots are either large or small.
\item There exist $0<\epsilon(s)=\epsilon(|s|)<\frac{1}{2}$, $0<\phi(s)<\frac{\pi}{\mu_-}$, and $\gamma(s)\in [0,2\pi)$ with the following properties:
\begin{itemize}
\item There is exactly one small root inside each connected component of $$Dart(\epsilon(s),\mu_-,\phi(s),\gamma(s))\subset \mathbb{A}_1.$$
\item As $s\to 0$, $\epsilon(s)$ and $\phi(s)$ converge to $0$.
\item $\gamma(s)$ is the argument of $s^{d_-}$ valued in $[0,2\pi).$ 
\item All small roots are simple. 
\end{itemize}

\item There exist $r(t)>0$ and $0<\phi'(t)<\frac{\pi}{\mu}$,  with the following properties:
\begin{itemize}
\item There is exactly one large root inside each connected component of $$Dart^{\infty}(r(t),\mu,\phi'(t))\subset \mathbb{A}_1.$$
\item As $t\to 0$, $r(t)\to\infty$ and $\phi(s)\to 0.$
\item All large roots are simple. 
\end{itemize}
\end{enumerate}
\end{cor}

\begin{proof} The statement about small roots is an immediate consequence of Proposition \ref{small-roots-prop}.
To deduce the statement about large roots we rewrite the Equation \eqref{critpol} in terms of the variable $u=1/z$ (equivalently, we consider solutions of Equation \eqref{proj-crit-pol-eq} in
the affine chart $\mathbb{A}_2$):
$$u^{\mu}=c(su^a+t)^{d_-}.$$
Now we observe that the small roots of this equation correspond to large roots of the equation in the affine chart $\mathbb{A}_1$,
and the assertion follows again from Proposition \ref{small-roots-prop}.
\end{proof}

\subsection{The vanishing spheres}\label{van-sph-sec}

In this section we will prove Theorem \ref{Da-H-thm}. Assume that $n>1$ (the case $n=1$ was discussed at the end of Section \ref{ss-Aside-intro}).

It will be convenient to analyze $\mathcal{D}_a$ inside $g_a^{-1}(0)$ instead of $g_a^{-1}(A)$ by dragging the regular point from $A$ to $0$ along a path that goes slightly below $\gamma_0$. Let us define the radial vanishing paths $r_1,\ldots r_{\mu(a)}$ in the base of $g_a$ as the straight radial paths from the critical values to the origin. 
They are ordered in the clock-wise direction and the last one in the ordering is the vanishing path of the positive real critical value. See Figure \ref{fig-move-A}. 
The directed Fukaya-Seidel $A_{\infty}$-category $\mathcal{E}_a$ of the Lagrangian vanishing spheres of $r_1,\ldots r_{\mu(a)}$ is quasi-isomorphic to $\mathcal{D}_a.$
 
\begin{figure}
\includegraphics[width=0.8\textwidth]{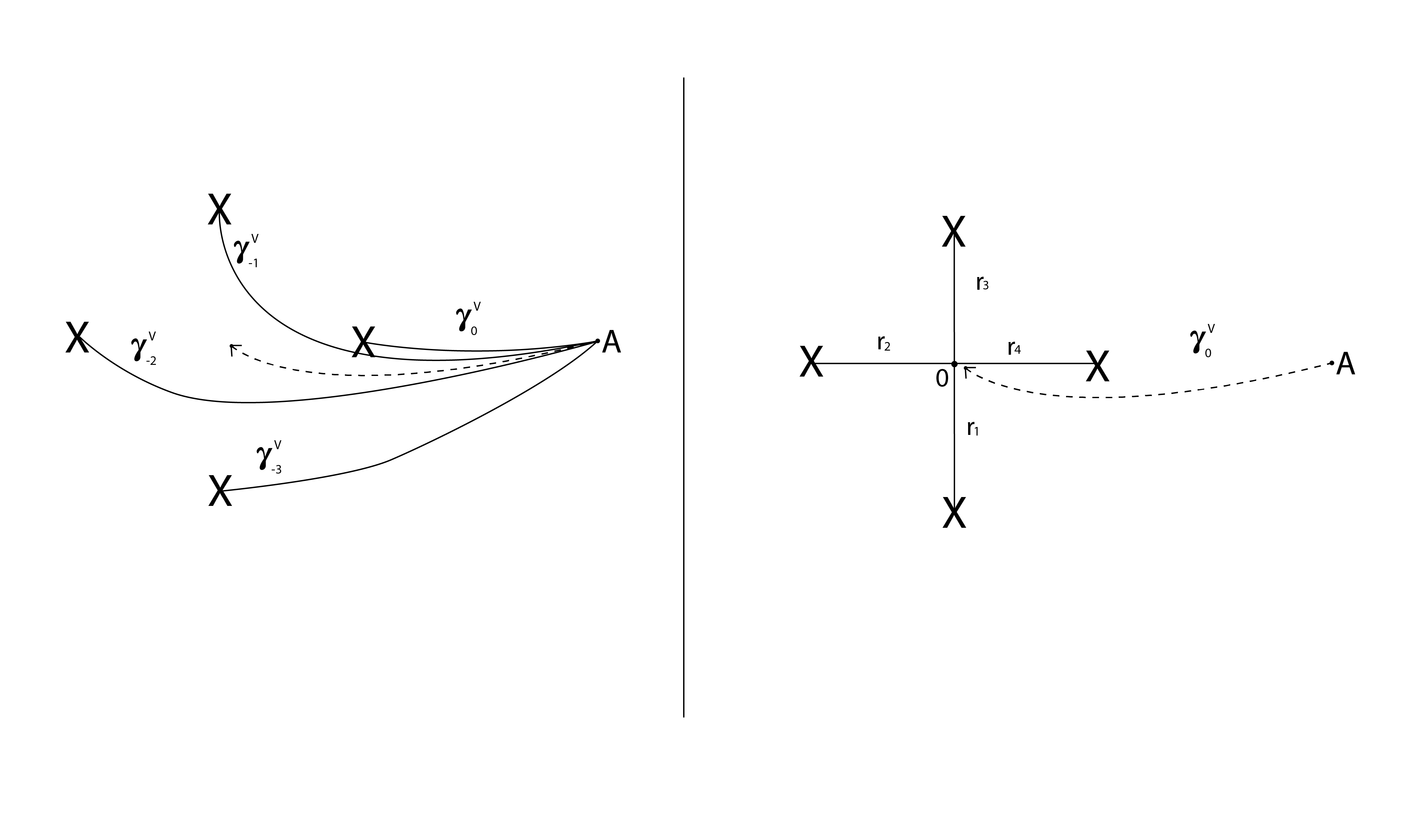}
\caption{Moving the dual vanishing paths to radial ones}
\label{fig-move-A}
\end{figure}

Let us define the map $g_a^{t,s}:\mathbb{C}^n\to\mathbb{C}$ by $$(z_1,\ldots ,z_n)\mapsto z_1-sz_2-tz_1^{a_1}z_2 -p_{-a}(z_2,\ldots z_n),$$ for complex numbers $t,s$. Note that $g_a^{1,0}=g_a$. Let us also note that for $t\neq 0$, \begin{equation}\label{eq-t-pert} g^{t,0}_a(z_1,\ldots,z_n)=\xi^{-d(a)}g_a(\xi^{\mu(a)+q_1}z_1,\xi^{q_2}z_2,\ldots, \xi^{q_n}z_n),\end{equation} where $\xi$ is a $(a_1\mu(a))^{th}$ root of $t$ and $$q_i=\mu(a_{i+1},\ldots,a_n)d(a_1,\ldots,a_{i-1})$$ are as in the Equation (3.2) of \cite{Seifert}. See Remark \ref{rmk-extra-perturb} for what lead us to consider the extra perturbation by $s$. 

\begin{lem}\label{lem-g-lef-small}
For every positive real number $t$, there exists a $\delta(t)>0$ such that  
\begin{itemize}
\item for every complex number $s$ with $|s|<\delta(t)$, $g_a^{t,s}$ is a Lefschetz fibration with $\mu(a)$ critical points; 
\item there exist $\mu(a)$ analytic maps $p_1,\ldots,p_{\mu(a)}:\C\to \C^n$ defined for $|s|<\delta(t)$ such that
$p_1(s),\ldots,p_{\mu(a)}(s)$ are exactly the critical points of $g_a^{t,s}$;
\item  if $d$ is the distance between the two closest critical values of $g_a^{t,0}$ 
then for $|s|<\delta(t)$,
each critical value of $g_a^{t,s}$ is contained in a $d/10$ neighborhood of a critical value  of $g_a^{t,0}$.
\end{itemize} 
\end{lem}
\begin{proof}
We already know that $g_a$ is a Lefschetz fibration with critical values regularly placed on a circle centered at the origin. Using Equation \eqref{eq-t-pert}, we see that the same statement is true for $g_a^{t,0}$ for $t\neq 0$, in particular for $t$ a positive real number.

From Equation \eqref{eq-t-pert} and the first paragraph of Section \ref{ss-fukaya-seidel} it follows that $g_a^{t,s}$ is tame for all complex numbers $t,s$. This implies that for fixed $t,s$ the critical points of $g_a^{t,s}$ are contained in a compact subset of $\mathbb{C}^n$ in the complex analytic topology. In fact the argument in Proposition 2.5 of \cite{Fan} shows that if we fix $t$ then there exists a compact subset $K\subset \mathbb{C}^n$ such that the critical points of $g_a^{t,s}$ are contained in $K$ if $|s|<1$.

Moreover, note that by non-degeneracy the natural scheme structure on the critical points of $g_a^{t,0}$ is smooth. 
Let us denote by $X$ the scheme of critical points of $g_a^{t,s}$ for fixed $t$ and varying $s$, so that we have
a projection $X\to \C_s$ and the fiber $X_s$ is the scheme of critical points of $g_a^{t,s}$.
Since the projection from $X$ to $\C_s$ is proper and $X_0$ is smooth, we deduce that the map $X\to \C_s$ is
\'etale over a small neighborhood of $0$. This implies the non-degeneracy of critical points of $g_a^{t,s}$ for small $s$. 
Also, it follows that there exist $\mu(a)$ analytic sections $p_1,\ldots,p_{\mu(a)}$ of the projection $X\to \C_s$ defined in
the neighborhood of $0$. This implies the second assertion. The last assertion follows from the fact that the critical
values $g_a^{t,s}(p_i(s))$, for $i=1,\ldots,\mu(a)$, depend continuously on $s$.
\end{proof}

Let us also define the maps $$h_a^{t,s}:(g_a^{t,s})^{-1}(0)\to \mathbb{C},$$ given by projecting to the $z_1$ coordinate.

We are going to compute all critical values of $h_a^{t,s}$. 
More generally, we will compute the critical values of $z_1$ on $(g_a^{t,s})^{-1}(y)$, for any regular value $y$ of $g_a^{t,s}$.

Let us set for brevity $g=g_a^{t,s}$.
Consider the family of maps 
$$w_y:g^{-1}(y)\to \mathbb{C},$$ 
for $y\in \mathbb{C}$, given by projecting to the $z_1$ coordinate (so $w_0=h_a^{t,s}$).

Let us define the Zariski closed subset $\CC\sub \C^n$ as the zero locus of $\partial_{z_2}g,\ldots,\partial_{z_n}g$.
Note that the tangent space to $g^{-1}(y)$ at a smooth point $z$ is given by the kernel $dg$, and 
that $z\in \mathcal{C}$ is a critical point of $w_y=z_1$ on $g^{-1}(y)$, where $y=g(z)$, if and only if
\begin{equation}\label{dz1-dg-eq}
dz_1|_z=\la\cdot dg|_z \ \text{ in } T^*_z\C^n,
\end{equation}
for some (necessarily nonzero) $\la\in\C$. In other words, for any $y\in \C$
we have
\begin{equation}\label{CC-crit-wy-eq}
\CC\cap g^{-1}(y)\setminus crit(g)=crit(w_y)\setminus crit(g).
\end{equation}

\begin{prop}\label{propcritvalueh} We fix $t,s\in\mathbb{C}$ and use the notation introduced above.

\noindent
(i) The map 
$$\mathcal{C}\to \C^2: z=(z_1,\ldots ,z_n)\mapsto (g(z), z_1)$$ 
induces a bijective morphism
$$\iota:\CC\to \CC',$$
where $\CC'\sub \C^2_{y,z_1}$ is the plane curve 
\begin{equation}\label{crit-val-wr-eq}
c_a(s+tz_1^{a_1})^{d(-a)}-(z_1-y)^{\mu(-a)}=0,
\end{equation}
where $c_a$ is some easily computable positive rational number. 
Furthermore, $\iota$ restricts to an isomorphism of algebraic varieties 
$\CC\setminus\iota^{-1}(S)\to \CC'\setminus S$, where
$$S=\{(y,z_1) \ |\ y=z_1, s+tz_1^{a_1}=0\}.$$ 
In other words, we have a well defined inverse morphism $\iota^{-1}:\CC'\setminus S\to \CC\setminus \iota^{-1}(S)$.

\noindent
(ii) For fixed $y$, which is not a critical value of $g$, the set of critical values of $w_y$
is exactly the set of roots $z_1$ of the equation \eqref{crit-val-wr-eq}. Furthermore,
the critical values of distinct critical points of $w_y$ are distinct. 
\end{prop}

\begin{proof}
(i) Let us write the equations defining $\CC\sub\C^n$:
\begin{align*}
s+tz_1^{a_1}=&a_2z_2^{a_2-1}z_3\\
z_2^{a_2}=&a_3z_3^{a_3-1}z_4\\
\ldots\\
z_{n-1}^{a_{n-1}}=&a_nz_n^{a_n-1}.
\end{align*}
Also, setting $y=g(z)$, we have
\begin{equation}\label{z1-y-g-eq}
z_1-y-sz_2-tz_1^{a_1}z_2=p_{-a}(z_2,\ldots z_n).
\end{equation}
Assuming that $(y,z_1,\ldots,z_n)$ satisfy these equations we have to show that $(y,z_1)$ satisfies \eqref{crit-val-wr-eq}
and that $(z_i)_{i\ge 2}$ are determined by $(y,z_1)$, and that for $(y,z_1)\notin S$, they are given by regular
functions $z_i(y,z_1)$.

If $s+tz_1^{a_1}=0$ then the equations of $\CC$ imply that $z_2=\ldots=z_n=0$, and the equation \eqref{z1-y-g-eq}
gives $z_1=y$, so that $(y,z_1)\in S$.


Now assume that $s+tz_1^{a_1}\neq 0$. Then we also have $z_i\neq 0$ for $i\ge 2$.
The last $n-1$ equations for $\CC$ lead to 
$$z_2^{a_2\mu(a_4,\ldots ,a_n)}=cz_3^{\mu(a_3,\ldots ,a_n)},$$ 
for a positive rational number $c$ that is straightforward to compute. 
Using the first equation for $\CC$ we get: 
\begin{align}\label{critvcom}(s+tz_1^{a_1})^{\mu(a_3,\ldots ,a_n)}=c'z_2^{\mu(a_2,\ldots ,a_n)}.\end{align}

Next, using equations for $\CC$, we can also obtain recursively for $k=2,\ldots,n-1$,
$$d(a_2,\ldots ,a_k)z_k^{a_k}z_{k+1}=sz_2+tz_1^{a_1}z_2.$$
Plugging this into the definition of $p_{-a}(z_2,\ldots,z_n)$ and then into \eqref{z1-y-g-eq}, we get 
$$z_1=c''(sz_2+tz_1^{a_1}z_2),$$ 
which leads to 
\begin{equation}\label{z1-z2-eq}
z_2=\frac{z_1}{c''(s+tz_1^{a_1})}.
\end{equation}
Plugging this into \eqref{critvcom}, we deduce the equation \eqref{crit-val-wr-eq} for $(y,z_1)$.

The desired formulas for $z_2,\ldots,z_n$ as rational functions of $(y,z_1)$ defined away from $S$,
are now easily obtained from \eqref{z1-z2-eq} and from the equations for $\CC$.

\noindent
(ii) In light of \eqref{CC-crit-wy-eq}, this follows from part (i).
\end{proof}

\begin{lem}\label{Lefschetz-lem}
Assume that $s\neq 0$ and $t$ is such that $0$ is not a critical value of $g=g_a^{t,s}$. Then all critical points of
$h_a^{t,s}$ on $g^{-1}(0)$ are nondegenerate.
\end{lem}

\begin{proof}
Let $z^0=(z^0_1,\ldots,z^0_n)$ be a critical point of $h_a^{t,s}$ on $g^{-1}(0)$.
Then $z^0$ belongs to $\CC$ and due to the relation \eqref{dz1-dg-eq}, we have 
$$\partial_1 g|_{z^0}=(1-a_1tz_1^{a_1-1}z_2)|_{z^0}\neq 0,$$
where we set $\partial_i=\partial_{z_i}$.
Thus, we can view $z_2,\ldots,z_n$ as local coordinates on $g^{-1}(0)$ near $z^0$
and compute the derivatives of $h=h_a^{t,s}=z_1$ with respect to $z_2,\ldots,z_n$ using the equation
$$h-th^{a_1}z_2=sz_2+p,$$ 
where $p=p_{-a}(z_2,\ldots,z_n)$.
This gives 
$$\partial_2h=\frac{s+th^{a_1}+\partial_2p}{1-a_1tz_1^{a_1-1}z_2},$$
$$\partial_ih=\frac{\partial_ip}{1-a_1tz_1^{a_1-1}z_2} \ \text{ for } i>2.$$
In particular, we have
$$(s+th^{a_1}+\partial_2 p)|_{z^0}=0, \ \ \partial_i p|_{z^0}=0 \ \text{ for } i>2.$$
Taking this into account we derive that for all $i,j\ge 2$,
$$\partial_i\partial_jh|_{z^0}=\frac{\partial_i\partial_jp}{1-a_1tz_1^{a_1-1}z_2}|_{z^0}.$$
Thus, it remains to show that the matrix $(\partial_i\partial_jp|_{z^0})_{i,j\ge 2}$ is invertible. 

First, we observe that $z_i^0\neq 0$ for $i=1,\ldots,n$. Indeed, as we have seen in the proof of Proposition
\ref{propcritvalueh}, the only other possibility is that all $z_i^0=0$, which is possible only when $s=0$
(due to equation \eqref{crit-val-wr-eq}).

Now our assertion follows from the following identity (applied to $p=p_{-a}$). For $a=(a_1,\ldots,a_n)$,
$$\Delta(a):=\det(\partial_i\partial_j p_a)_{1\le i,j\le n}.$$
Then at any point $z$ where $\partial_i p_a=0$ for $i>1$, one has
$$\Delta(a)=(-1)^{{n+1\choose 2}}\cdot r\cdot z_1^{a_1-2}z_2^{a_2-1}\ldots z_n^{a_n-1}$$
with $r>0$.
Indeed, this can be checked easily by induction since
\begin{align*}
&\Delta(a)=-a_1(a_1-1)z_1^{a_1-2}z_2\cdot (-1)^{n-1}\Delta(-a)-a_1^2z_1^{2a_1-2}\cdot \Delta(--a)=\\
&a_1(a_1-1)z_1^{a_1-2}z_2\cdot (-1)^n\Delta(-a)-a_1^2a_2z_1^{a_1-2}z_2^{a_2-1}\cdot \Delta(--a),
\end{align*}
where $-a=(a_2,\ldots,a_n)$, $--a=(a_3,\ldots,a_n)$ (we used the equation $z_1^{a_1}=a_2z_2^{a_2-1}z_3$).
\end{proof}

Recall Corollary \ref{cor-dart} and Lemma \ref{lem-g-lef-small}. Let us fix $t_0$ and $s_0$, positive real numbers with $$t_0<C(a_1,\mu(-a),d(-a),c_a)$$ and $$s_0<\min{\{C(a_1,\mu(-a),d(-a),c_a),\delta(t_0)\}}.$$ 

In the base of $g_a^{t_0,s_0}$, we consider the radial vanishing paths $$\tilde{r}_1,\ldots \tilde{r}_{\mu(a)}$$ from each of the critical values to the origin. These are again ordered clockwise and so that $\tilde{r}_{\mu(a)}$ aligns with the positive real axis.

\begin{rmk} Note that $g_a^{t_0,s_0}$ indeed has a unique positive real critical value. This follows because we know that the only critical value of $g_a^{t,0}$ whose $d/10$ neighborhood intersects the positive real axis is the positive real one and that the set of critical values of $g_a^{t_0,s_0}$ is closed under complex conjugation of $\mathbb{C}.$ Therefore, $\tilde{r}_{\mu(a)}$ still aligns with the positive real line.
\end{rmk}

The directed Fukaya-Seidel $A_{\infty}$-category $\tilde{\mathcal{E}}_a$ of the Lagrangian vanishing spheres of $\tilde{r}_1,\ldots \tilde{r}_{\mu(a)}$ is equivalent (as a directed $A_{\infty}$-category) to
$\mathcal{E}_a$, and therefore, to $\mathcal{D}_a.$ 

Our goal is to compute the Lagrangian vanishing spheres of $\tilde{r}_1,\ldots \tilde{r}_{\mu(a)}$ as Lagrangian matching spheres inside $(g_a^{t_0,s_0})^{-1}(0)$ corresponding to matching paths in the base of $h_a^{t_0,s_0}.$

By Proposition \ref{propcritvalueh} (ii), the critical values of $h_a^{t_0,s_0}$ 
are solutions of the equation 
$$z_1^{\mu(-a)}=c_a(s_0+t_0z_1^{a_1})^{d(-a)},$$ 
Moreover, 
the critical values of distinct critical points of $h_a^{t_0,s_0}$ are not equal to each other.
These critical values are divided into two groups:

\begin{itemize}
\item small ones: one in each connected component of an \emph{inner dart} $Dart(\epsilon(s_0),\mu(-a),\phi(s_0),0)$
\item large ones: one in each connected component of an \emph{outer dart} $Dart^{\infty}(r(t_0),\mu(a),\phi'(t_0))$. 
\end{itemize} Note that $\epsilon(s_0)<1/2$ and $r(t_0)>2$.\\

Recall that we have defined in the introduction the symplectomorphism $\psi_a$ of $\mathbb{C}^n$ which gives the action of the element of 
$\Gamma_a$ with $\lambda=\lambda_1=e^{\frac{2\pi i}{\mu(a)}}$.

\begin{lem}\label{theta-rescale-lem} 
We have the following commutative diagram \begin{align}
\xymatrix{ 
\mathbb{C}^n\ar[r]^{\psi_a}\ar[d]_{g_a^{t,s}}& \mathbb{C}^n \ar[d]^{g_a^{t,e^{i\theta}\cdot s}}\\ \mathbb{C}\ar[r]_{\mathrm{rot}_{2\pi/\mu(a)}}&\mathbb{C},}
\end{align} with $$\theta=\frac{2\pi a_1}{\mu(a)}.$$
\end{lem}
\begin{proof}
This is a straightforward computation.
\end{proof}

Recall that for $\gamma\in [0,2\pi)$, we denote by $R_{\gamma}$
the ray in the complex plane starting from the origin that makes a positive angle of $\gamma$ with the positive real axis. 

\begin{prop}\label{vanishing-Lagrangian-sphere-matching-prop}
Let $\varphi:= \frac{2\pi k}{\mu(a)}$ for some $k=0,\ldots , \mu(a)-1$.
Consider $$s=s_0e^{\frac{2\pi ika_1}{\mu(a)}}.$$ 
We have:

\begin{itemize}
\item $g_a^{t_0,s}$ and $h_a^{t_0,s}$ are Lefschetz fibrations.
\item $g_a^{t_0,s}$ has a unique critical value $b$ on $R_{\varphi}$.
\item The map $h_a^{t_0,s}$ has precisely two critical values $b_1$, $b_2$ on $R_{\varphi}$.
\item The vanishing Lagrangian sphere of the straight vanishing path from $0$ to $b$ is Hamiltonian isotopic to the matching Lagrangian sphere of the matching path between $b_1$ and $b_2$
along $R_{\varphi}$. In particular, this straight path is a matching path.
\end{itemize}
\end{prop}

\begin{proof} 
By Lemma \ref{theta-rescale-lem}, it suffices to proves this for $k=0$. 

By the choice of $s_0$, $g_a^{t_0,s_0}$ is a Lefschetz fibration. Also, by
Lemma \ref{Lefschetz-lem},  $h_a^{t_0,s_0}$ is Lefschetz fibration. 

That $g_a^{t_0,s_0}$ has a unique critical value on the positive real axis was already remarked above. The proof that $h_a^{t_0,s_0}$ has precisely two critical values on the positive real axis follows exactly the same strategy. We know that the unique connected component of both $Dart(\epsilon(s_0),\mu(-a),\phi(s_0),0)$ and $Dart^{\infty}(r(t_0),\mu(a),\phi'(t_0))$ that intersect the positive real axis contain exactly one critical value and that they are preserved under complex conjugation.

We come to the last bullet point. This is a simple application of the Lefschetz bifibration technique. 
Let us denote the unique positive real critical value of $g=g_a^{t_0,s_0}$ by $b$. 

We claim that for $(y,z_1)\in \CC'$ the map 
$pr_1|_{\mathcal{C}'}:\mathcal{C}'\to \mathbb{C}$ is \'etale at $(y,z_1)$ (i.e., induces an isomorphism of tangent
spaces, so in particular, $\mathcal{C}'$ is smooth at this point) unless $(y,z_1)\in S$ and $\iota^{-1}(y,z_1)$ is
a critical point of $g$. Indeed, first, one can immediately check that
$pr_1$ is unramified at the points of $S\subset \CC'$.
Thus, it is enough to check that the map $g=pr_1\circ\iota:\CC\to \C$ is unramified at all $z\not\in (crit(g)\cup \iota^{-1}(S))$.
Indeed, let $T^*_z\CC$ denote the Zariski cotangent space to $\CC$ at any such point $z$. Since
$(g,z_1)$ gives an embedding of $\mathcal{C}\setminus \iota^{-1}(S)$ into $\C^2$, $T^*_z\CC$ is generated by
the images of $dz_1|_z$ and $dg|_z$. 
Now from \eqref{dz1-dg-eq} we see that in fact $T^*_z\CC$ is generated by $dg|_z$ alone. This implies that 
$\dim T^*_z\CC=1$, so $\CC$ is smooth, and the tangent map to $g$ is an isomorphism at $z$ as claimed.



As a consequence, if $r\in [0,\infty)$ is so that the equation \eqref{crit-val-wr-eq} for $y=r$
has a positive real root $z_1$ with multiplicity more than one, then 
$\iota^{-1}(y,z_1)$ is a critical point of $g_a^{t_0,s_0}$ $(\star)$. In particular, this can only happen for $r=b$.

Consider the Equation \eqref{crit-val-wr-eq} for $y=r\in [0,\infty)$. We already know that for $r=0$ there are two simple positive real roots. It is also easy to see that for $r$ sufficiently large, there are no positive real roots that are larger than $r$. Also note that $z_1=r$ is never a root. Combining these with the previous paragraph, we conclude that the two positive real roots at $r=0$ come together on the positive real axis for the first time at $r=b.$\footnote{It also follows that for larger values of $r$ there is never a real root larger than $r$. Note that we are not claiming that are no other positive roots, we only consider the positive roots that are larger than $r$ in this argument.}

Moreover, using $\star$ from two paragraphs ago, it follows that the critical points above the two positive real critical values of $w_r$ (as elements of $\mathbb{C}^n$) come together at the unique singular point $p=(p_1,\ldots, p_n)$ of $(g_a^{t_0,s_0})^{-1}(b)$ as $r$ goes from $0$ to $b$.

Instead of proving that \begin{align}
\xymatrix{ 
\mathbb{C}^n\ar[rr]^{(g_a^{t_0,s_0}, z_1)}&& \mathbb{C}^2 \ar[r]^{pr_1}&\mathbb{C}}
\end{align} is a Lefschetz bifibration, we will prove that the there are coordinates near  $p\in \mathbb{C}^n$, $(b,p_1)\in \mathbb{C}^2$ and $b\in \mathbb{C}$ as in Lemma 15.9 of \cite{Seidelbook}.

We first find coordinates as in equation in the last line of pg 219 in \cite{Seidelbook} using the argument given there. On $\mathbb{C}^2$ and $\mathbb{C}$ we use the given coordinates on this step. All we need to prove is that the map $\mathbb{C}^{n-1}\to \mathbb{C}$ obtained by substituting $z_1=p_1$ in  $g_a^{t_0,s_0}$ has a non-degenerate singularity at $p$. This map is given by $$p_1-s_0z_2-t_0p_1^{a_1}z_2 -p_{-a}(z_2,\ldots z_n).$$

Note that since $s_0,t_0$ and $p_1$ are all positive real numbers $$s_0+t_0p_1^{a_1}\neq 0.$$ Therefore, we know that $-(s_0+t_0p_1^{a_1})z_2 -p_{-a}(z_2,\ldots z_n)$ has only non-degenerate critical points, proving our claim.

To finish finding the desired local coordinates, we can repeat the part of the proof of Lemma 15.9 of \cite{Seidelbook} on pg 220 verbatim since we know that $p$ is a non-degenerate critical point of $g_a^{t_0,s_0}.$

Hence, using Lemma 16.15 of \cite{Seidelbook}, we conclude that the path between the  two positive real critical values of $w_0=h_a^{t_0,s_0}$ is a matching path and the matching Lagrangian sphere above is Hamiltonian isotopic to the vanishing Lagrangian sphere of the straight path from the origin to $b$ in the base of $g_a^{t_0,s_0}$.
\end{proof}

\begin{rmk}
Note that we never proved  that our $\mathbb{C}^n\to \mathbb{C}^2\to\mathbb{C}$ is a Lefschetz bifibration, which would require checking a number of non-degeneracy requirements as explained in page 218 of \cite{Seidelbook}.
\end{rmk}

Let $\mathbb{A}:=\{x\in \mathbb{C} | 1/2\leq |x|\leq 2\}.$ We define a diffeomorphism  $coil_a:\mathbb{A}\to\mathbb{A}$, which is in  polar coordinates $$(\rho,\theta)\mapsto (\rho,\theta+f(\rho)),$$ where $f(\rho)$ is non-decreasing in $\rho$, equal to $-\frac{2\pi}{\mu(-a)}$ near $\rho=1/2$, and equal to $\frac{2\pi}{\mu(a)}$ near $\rho= 2.$

Recall that we have the matching path in the base of $h_a^{t_0,s_0}$ which is 
the straight line segment $[b_1,b_2]$ connecting the two positive critical values $b_1<b_2$. 
Now for $t$, $0<t\le t_0$, and $k=0,\ldots,\mu(a)-1$, we will define a path $\sigma(k,t)$ connecting two critical values of $h_a^{t,s_0}$ in $\mathbb{C}$. 
Note that $b_1<1/2$ and $b_2>2$.

We apply $coil_a^{\circ k}$ to $[b_1,b_2]\cap \mathbb{A}=[1/2,2]$ and obtain a path from $p_1=coil_a^{\circ k}(1/2)$ to $p_2=coil_a^{\circ k}(2)$ in $\mathbb{A}$. 
Then we connect $p_1$ to a point $p'_1\in Dart(\epsilon(s_0),\mu(-a),\phi(s_0),0)$ and $p_2$ to a point $p'_2\in Dart^{\infty}(r(t),\mu(a),\phi'(t))$ by radial paths.
Finally, we connect $p'_1$ (resp., $p'_2$) by a smooth path depending smoothly on $t$
to a critical value without leaving $Dart(\epsilon(s_0),\mu(-a),\phi(s_0),0)$ (resp., $Dart^{\infty}(r(t),\mu(a),\phi'(t))$). See Figure \ref{figdart}.
These paths together form the path we call $\sigma(k,t)$.
\begin{figure}
\includegraphics[width=0.5\textwidth]{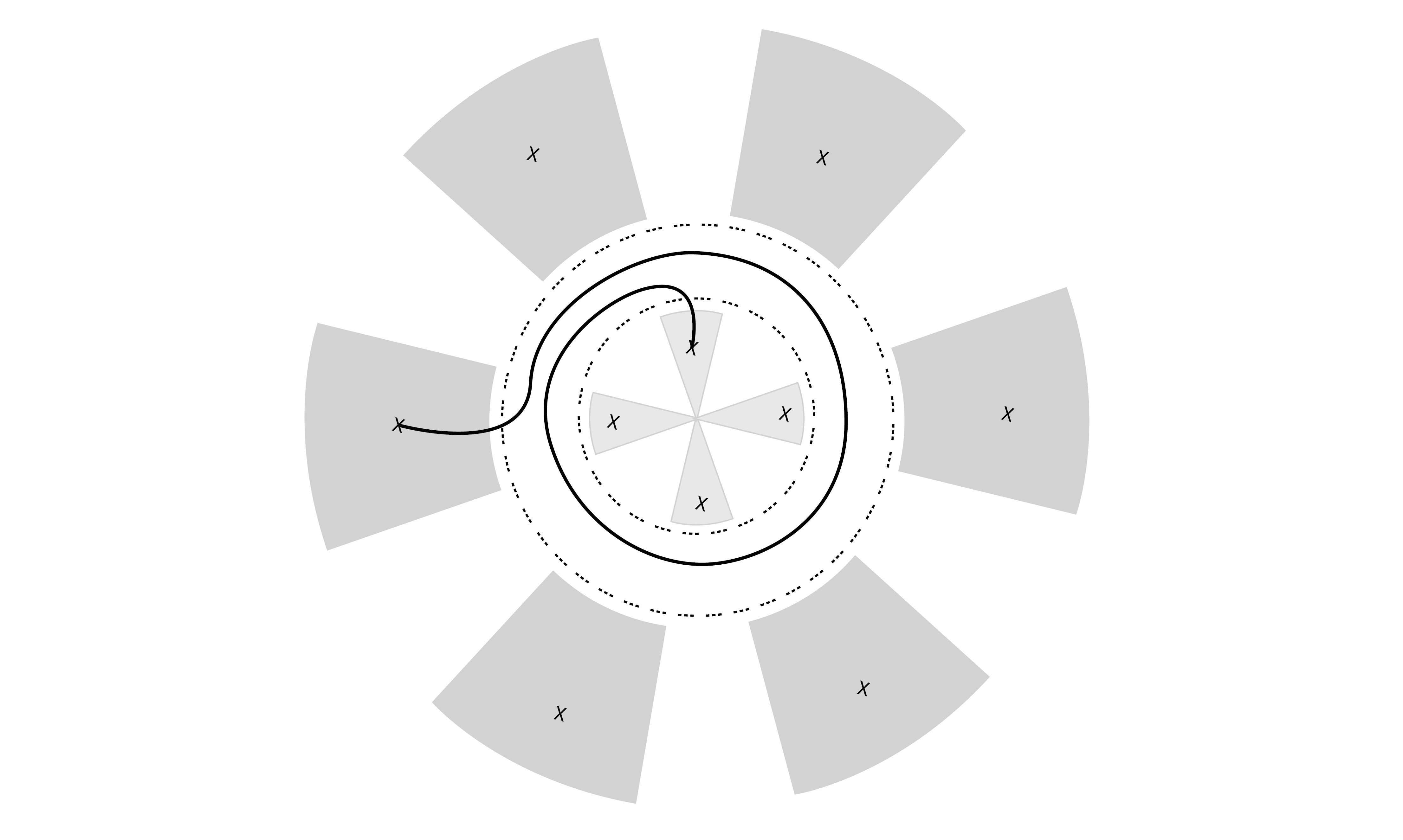}
\caption{The coiling matching paths}
\label{figdart}
\end{figure}

Let us denote by $C_0(\varphi)$ (resp., $C_\infty(\varphi)$) the component of
$Dart(\epsilon(s_0),\mu(-a),\phi(s_0),0)$ (resp., $Dart^{\infty}(r(t),\mu(a),\phi'(t))$) centered around the ray with argument $\varphi$.
Note that $\sigma(k,t)$ connects a critical value of $h^{t,s_0}$ in $C_0(-2\pi \frac{k}{\mu(-a)})$ with a critical value of $h^{t,s_0}$ in $C_\infty(2\pi \frac{k}{\mu(a)})$.

We set $$\sigma(k):=\sigma(k,t_0).$$ 

\begin{prop}\label{prop-match-final}
The vanishing spheres of $\tilde{r}_1,\ldots \tilde{r}_{\mu(a)}$ are Hamiltonian isotopic to the Lagrangian matching spheres of the paths $$\sigma(\mu(a)-1),\ldots ,\sigma(1), \sigma(0)$$ 
in the base of $h_a^{t_0,s_0}$.

\end{prop}

\begin{proof}
By Proposition \ref{vanishing-Lagrangian-sphere-matching-prop},
for every $k=0,\ldots , \mu(a)-1$, we can compute the vanishing Lagrangian sphere of the critical value of $g^{t_0,s}_a$ with argument $\varphi:= \frac{2\pi k}{\mu(a)}$ as the matching Lagrangian sphere of an explicit straight matching path $\beta$ connecting two critical values of $h^{t_0,s}_a$ on the ray $R_\varphi$ for $$s=s_0e^{\frac{2\pi ika_1}{\mu(a)}}.$$


Let us call an embedded path in the base of Lefschetz fibration with endpoints on critical values and interiors disjoint from critical values a pre-matching path. Note that a matching path is in particular a pre-matching path.

Now all we need to do is to prove that there exists a smoothly varying family of pre-matching paths $\beta_\tau$, where $\tau$ varies in $[0,1]$, in the bases of $h^{t_0,s_\tau}_a$ for 
$$s_\tau=s_0e^{\frac{2\pi ika_1(1-\tau)}{\mu(a)}},$$ such that\begin{itemize}
\item $\beta_0$ is isotopic to $\beta$ through pre-matching paths;
\item $\beta_1$ is isotopic to $\sigma(k)$ through pre-matching paths.
\end{itemize}

By Corollary \ref{cor-dart}, we know that for every $\tau\in [0,1]$, the critical values of $h^{t_0,s_\tau}_a$ are divided into small ones and big ones and into sectors as follows:
\begin{itemize}
\item there is one critical value $b_m(\tau)$ in each connected component $$\rot_{(1-\tau)\th_k} C(2\pi\frac{m}{\mu(-a)}),$$ $m=0,1,\ldots, \mu(-a)-1$, of $Dart(\epsilon(s_\tau),\mu(-a),\phi(s_\tau),\mu(-a)(1-\tau)\th_k)$,
where $$\th_k=\frac{2\pi ka_1d(-a)}{\mu(-a)\mu(a)};$$
\item there is one critical value $B_p(\tau)$ in each connected component $C_\infty(2\pi\frac{p}{\mu(a)})$, $p= 0,1,\ldots, \mu(a)-1$, of $Dart^{\infty}(r(t_0),\mu(a),\phi'(t_0))$. 
\end{itemize}
In particular that there is never any critical value in the annulus $\mathbb{A}$. To summarize in words these two bullet points, as $\tau$ changes from $0$ to $1$, the component containing $b_m(\tau)$ (for a fixed $m$) rotates clockwise with the angular velocity $\theta_k$, while the components containing $B_p(\tau)$ do not move (although the critical points can move inside them).

Note that we have the identity 
$$2\pi\frac{k}{\mu(a)}=-2\pi\frac{k}{\mu(-a)}+\th_k,$$ 
which shows that the straight path $\beta$ connects the small critical value $b_{-k}(0)$ with the big critical value $B_k(0)$.


Here is how we define $\beta_\tau$. First, we connect the critical value $b_{-k}(\tau)$ by a radial path with a point $q_1(\tau)$ lying on the circle of radius $1/2$.
Similarly, we connect the critical value $B_k(\tau)$ by a radial path with a point $q_2(\tau)$ lying on the circle of radius $2$.

To continue let us introduce the isotopy $\eta_\tau:\mathbb{A}\to\mathbb{A}$, which is in  polar coordinates $$(\rho,\theta)\mapsto (\rho,\theta+\tau f(\rho)),$$ where $f(\rho)$ is non-decreasing in $\rho$, equal to 
$-\th_k$ near $\rho=1/2$, and equal to $0$ near $\rho= 2.$

We finally connect $q_1(\tau)$ with one end of $\eta_\tau(\beta\cap \mathbb{A})$ using the short arc on the circle of radius $1/2$ and $q_2(\tau)$ with the other end of $\eta_\tau(\beta\cap \mathbb{A})$ using the short arc on the circle of radius $2$. This completes the construction of pre-matching paths $\beta_\tau$ with the desired properties. 

\end{proof}

\begin{rmk}
Note that the braid monodromy of the bases of $h^{t_0,s}_a$ as $s=e^{i\theta}s_0$ makes a full counter clock-wise rotation is non-trivial and can be easily computed using the proof above. Noting that the Hamiltonian fiber bundle over $S^1$ of the total spaces of these fibrations is actually trivial (it extends to a fiber bundle over the disk that bounds the $S^1$), we can generate lots of matching paths in the base of $h^{s_0,t_0}_a$ with Hamiltonian isotopic Lagrangian matching spheres.
\end{rmk}

\begin{figure}
\includegraphics[width=0.7\textwidth]{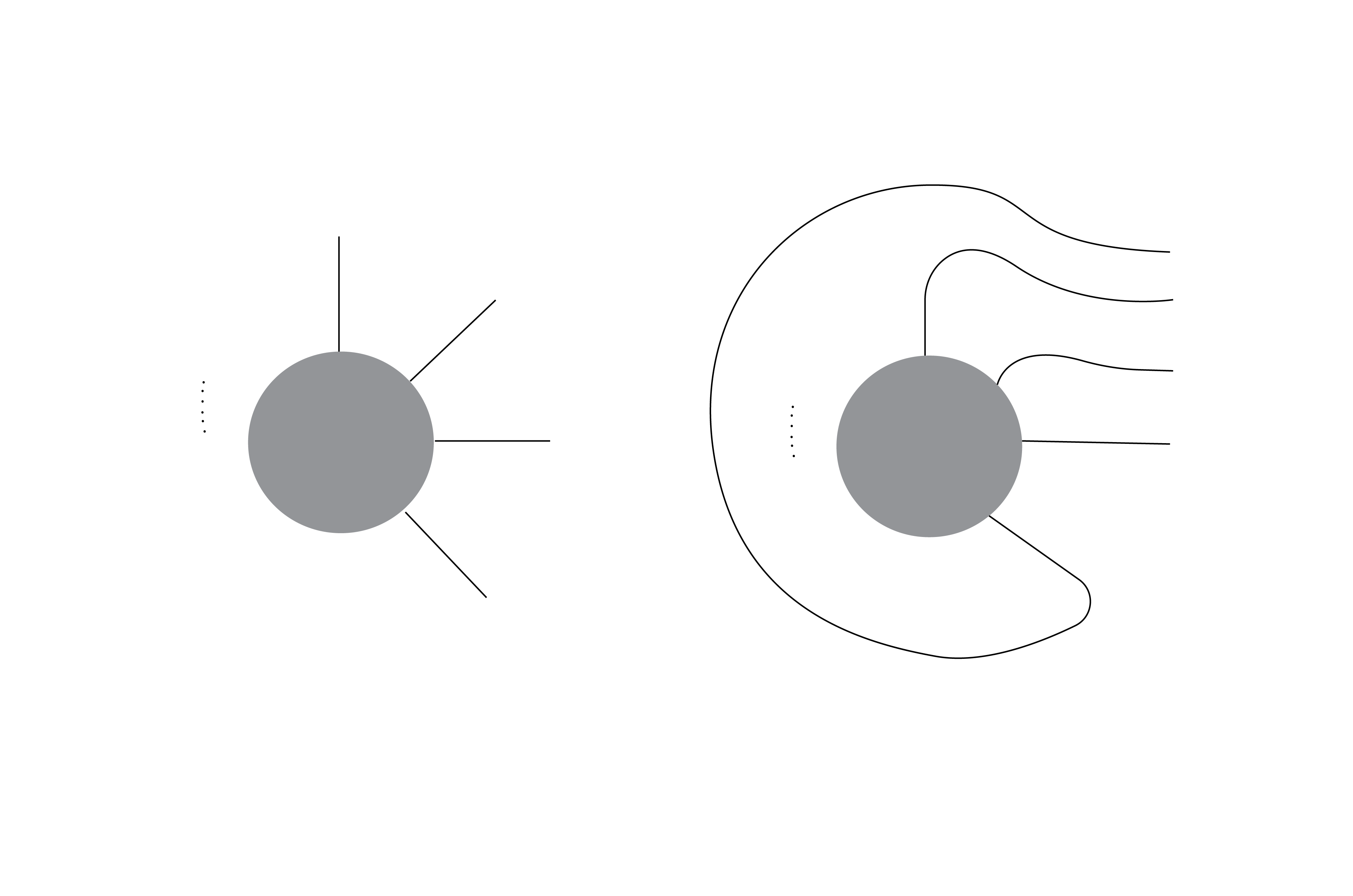}
\caption{Turning radial at infinity paths into horizontal at infinity paths}
\label{fradhor}
\end{figure}

We come to the final geometric argument of our proof. For $t\in[0,t_0]$, consider the family of Lefschetz fibrations $h_a^{t,s_0}$. For all values of $t$ in this interval the critical values stay in the same darts from $h_a^{t_0,s_0}$, but as $t\to 0$, the large roots go to infinity in a very controlled way.

For every $t\in[0,t_0]$, we have a directed $A_{\infty}$-category called $C_t$, which is the directed $A_{\infty}$-category of the matching paths of 
$$\si(\mu(a)-1,t),\ldots ,\si(1,t), \si(0,t)$$ for $t>0$. For $t=t_0$, by Proposition \ref{prop-match-final}, we have an equivalence of directed $A_\infty$-categories
$C_{t_0}\simeq \wt{\EE}_a$.
On the other hand, for $t=0$ we get a directed collection of vanishing paths, radial at infinity.
These can be homotoped to HAI vanishing paths with strictly decreasing ordinals,
$$\wt{\si}(\mu(a)-1),\ldots,\wt{\si}(0),$$
in the way that is explained in Figure \ref{fradhor} without changing them inside the disk of radius $2$.


Note that we can choose perturbations so that for all $t$, all the solutions of the perturbed pseudo-holomorphic curve equations that contribute to the structure maps of $C_t$ lie inside $(h_a^{t,s_0})^{-1}(\mathbb{D})$, where $\mathbb{D}$ is the disk of radius $2$, by the open mapping principle.

Using the homotopy method (Section (10e) of \cite{Seidelbook}), we obtain the following statement

\begin{prop}
There is an $A_{\infty}$ quasi-isomorphism $C_{t_0}\rTo{\sim} C_0$ preserving the ordering of the objects.
\end{prop}

\begin{rmk}\label{rmk-homotopy-method}
Our setup is slightly different than Seidel. Our Lagrangians are not known to be pair-wise transverse at all times, so one does have to consider possible birth-death bifurcations. It might be possible to use more geometry to prove that the Lagrangians we are considering are already transverse or that one can find smoothly varying Hamiltonian perturbations that achieve this property, but we did not check this. Regardless, we do not expect a problem with the birth-death analysis because we are considering directed $A_{\infty}$-categories in the exact case. The skeptic reader can assume this proposition not proven.\end{rmk}

The following proposition finishes the proof of Theorem \ref{Da-H-thm}.

\begin{prop}
\begin{itemize}
\item $\mathcal{D}_a$ is equivalent (up to shifts) to $C_{t_0}$ as directed $A_{\infty}$-categories.
\item $\mathcal{H}_{-a}$ is equivalent (up to shifts) to $C_{0}$ as directed $A_{\infty}$-categories.
\end{itemize}
\end{prop}
\begin{proof}
We already know the first statement: $\DD_a\simeq \EE_a\simeq \wt{\EE}_a\simeq C_{t_0}$.

For the second one, note that we have 
$$g_a^{0,s_0}(z_1,\ldots,z_n)=z_1-s_0z_2-p_{-a}(z_2,\ldots,z_n)=z_1-\wt{g}(z_2,\ldots,z_n),$$
where  $\wt{g}(z_2,\ldots,z_n)=s_0z_2+p_{-a}(z_2,\ldots,z_n)$ is a perturbation of $p_{-a}$, which is equivalent to the perturbation
$z_2+p_{-a}(z_2,\ldots,z_n)$.
Thus, $h_a^{0,s_0}$ is nothing but the projection from the graph of $\wt{g}$,
$$\{z_1=\wt{g}(z_2,\ldots,z_n)\}\subset \mathbb{C}^n$$ 
to the $z_1$ coordinate, which can be identified with $\wt{g}:\C^{n-1}\to \C$. 
The map on the total spaces is not a symplectomorphism but the induced Ehresmann connections do go to each other, which is enough for our purposes.
It remains to observe that the collection of HAI vanishing paths $(\wt{\si}(\mu(a)-1),\ldots,\wt{\si}(0))$ is homotopic to 
$(\wt{\ga}_{-\mu(a)+1},\ldots,\wt{\ga}_{-1},\wt{\ga}_0)$ (see \eqref{tw-gamma-collection}).
\end{proof}

\section{B-side}\label{s-B-side}

\subsection{Semiorthogonal decompositions, exceptional collections and mutations}
\label{exc-coll-basics-sec}

For the most part, on the B-side we can work at the level of triangulated categories, without using dg-enhancements.
However, we will use existence of dg-liftings of some adjoint functors.
Namely, by the results of \cite[Sec.\ 4]{KuzLunts},
 if $\DD$ is an enhanced triangulated category and $\CC\sub \DD$ is an admissible subcategory, then
with respect to the induced dg-enhancement on $\CC$, the left and right adjoint functors $\la,\rho:\DD\to\CC$
can be lifted to quasi-functors between the corresponding dg-categories. We will tacitly use such liftings below
in the results that use the dg-enhancements.

Given an admissible subcategory $\CC\sub\DD$, we define the functor
of {\it left mutation through} $\CC$, 
$$L_{\CC}:{}^{\perp}\CC\to\CC^\perp$$
by the exact triangle
$$C\to X\to L_{\CC}(X)\to\ldots$$
Note that $L_\CC$ is just the restriction to ${}^{\perp}\CC$ of the left adjoint functor to the inclusion of $\CC^\perp$.

This definition has the following transitivity property. Suppose $\CC_1,\CC_2\sub \DD$ is a pair of admissible subcategories such that
$\Hom(\CC_2,\CC_1)=0$. Then the subcategory $\lan \CC_1,\CC_2\ran\sub \DD$ is also admissible and
$$L_{\lan \CC_1,\CC_2\ran}\simeq L_{\CC_1}\circ L_{\CC_2}|_{{}^{\perp}\lan \CC_1,\CC_2\ran}.$$

Similarly, the functor of {\it right mutation through} $\CC$, 
$$R_{\CC}:\CC^\perp\to{}^{\perp}\CC$$
is defined by the exact triangle
$$R_{\CC}(X)\to X\to C\to\ldots$$
One can immediately see that $R_{\CC}$ and $L_{\CC}$ are mutually inverse equivalences.

\begin{lem}\label{adjoints-mutation-lem}
Let $\CC\sub\DD$ be an admissible subcategory, and let $\la,\rho:\DD\to \CC$ denote the left and right adjoint functors to the inclusion. 
Then for $X\in {}^{\perp}\CC$,
one has a functorial isomorphism
$$\rho(X)\simeq \la(L_{\CC}(X)[-1]).$$
\end{lem}

\Pf . By definition, there is an exact triangle
$$L_{\CC}(X)[-1]\to C\to X\to L_{\CC}(X) $$
with $C\in \CC$, and we have $X\in {}^\perp\CC$, $L_{\CC}(X)\in\CC^\perp$.
This immediately implies that 
$$C\simeq \rho(X)\simeq \la(L_{\CC}(X)[-1]).$$
\ed

For an exceptional object $E$ we set $L_E:=L_{\lan E\ran}$, where $\lan E\ran$ is the admissible
subcategory generated by $E$.

\begin{definition}
Let $E_0,\ldots,E_n$ be an exceptional collection.
The {\it left dual} exceptional collection to $E_0,\ldots,E_n$ is the unique full exceptional collection $F_{-n},\ldots,F_0$ in $\lan E_0,\ldots,E_n\ran$
with $\Hom^*(E_i,F_{-j})=0$ for $j\neq i$ and $\Hom^*(E_i,F_{-i})=\k[0]$.
In fact, one has $F_0=E_0$ and for $i>0$,
$$F_{-i}=L_{E_0}\ldots L_{E_{i-1}}E_i.$$
In this situation we also say that $E_0,\ldots,E_n$ is the {\it right dual} exceptional collection to $F_{-n},\ldots,F_0$.
\end{definition}

\begin{lem}\label{block-dual-coll-lem}
Let $\DD=\lan \CC,\CC'\ran$ be a semiorthogonal decomposition. Let $E_0,\ldots,E_n$ (resp., $E'_0,\ldots,E'_m$) be an exceptional collection generating $\CC$
(resp., $\CC'$), and let $F_{-n},\ldots,F_0$ (resp., $F'_{-m},\ldots,F'_0$) be the left dual exceptional collection.
Then the exceptional collection
$$L_{\CC}(F'_{-m}),\ldots,L_{\CC}(F'_0),F_{-n},\ldots,F_0$$
is left dual to $E_0,\ldots,E_n,E'_0,\ldots,E'_m$.
\end{lem}

\subsection{Serre functor and helices}

We have the following well known connection between the Serre functor and mutations.

\begin{lem}\label{Serre-mutation-lem} 
Let $E$ be an exceptional object in $\DD$ and let $\CC=\lan E\ran^{\perp}$, so that we have a semiorthogonal decomposition
$$\DD=\lan \CC, \lan E\ran\ran.$$
Then there is an isomorphism
$$\SS_\DD(E)\simeq L_{\CC}(E).$$
\end{lem}

\begin{definition} Let $E_1,\ldots,E_n$ be  an exceptional collection generating the category $\CC$.
The {\it helix} generated by this exceptional collection is the sequence of exceptional objects
$(E_i)_{i\in \Z}$, extending $(E_1,\ldots,E_n)$, such that $\SS_{\CC}E_i=E_{i-n}$, where $\SS_{\CC}$ is the Serre functor of $\CC$.
\end{definition}

By Lemma \ref{Serre-mutation-lem}, we see that in a helix we have
$$E_i\simeq L_{E_{i+1}}\ldots L_{E_{i+n-1}}E_{i+n}.$$

\subsection{Aramaki-Takahashi exceptional collection}\label{AT-exc-coll-sec}

\subsubsection{Basic definitions}\label{AT-basic-sec}

Recall that for $a=(a_1,\ldots,a_n)\in \Z_{>1}^n$ we consider the chain polynomial
$$p_a=x_1^{a_1}x_2+x_2^{a_2}x_3+\ldots+x_{n-1}^{a_{n-1}}x_n+x_n^{a_n}.$$

Recall that we set $d(a)=a_1a_2\ldots a_n$ (with $d(\emptyset)=1$)
and we have the recursion for the Milnor numbers
$$\mu(a)=d(a)-\mu(-a)=a_1\ldots a_n-a_2\ldots a_n+a_3\ldots a_n-\ldots$$
(with $\mu(\emptyset)=1$). 
Let us set
$$\mu^{\vee}(a):=\mu(a^\vee),$$
so that
$$\mu^{\vee}(a)=d(a)-\mu^{\vee}(a-).$$ 

We denote by $L=L_{a}$ the maximal grading group for which $p=p_a$ is homogeneous, i.e., the abelian group with generators
$\ov{x}_i$, $\ov{p}$ and defining relations
$$a_1\ov{x}_1+\ov{x_2}=a_2\ov{x}_2+\ov{x_3}=\ldots=a_n\ov{x_n}=\ov{p}.$$
Note that the quotient $L/(\ov{p})$ is a cyclic group of order $d(a)$, 
generated by the image of $\ov{x}_1$, so that we have an exact sequence
$$0\to \Z\rTo{\ov{p}} L\to \Z/d(a)\to 0.$$
It will be convenient for us to set
$$\tau=(-1)^n\ov{x}_1.$$ 
 
By a graded matrix factorization of $p_a$ we always mean $L$-graded matrix factorizations, or equivalently $\Ga$-equivariant matrix factorizations,
where $\Ga=\Ga_a$ is the subgroup of $\G_m^n$ that has $L$ as the character group. 
 
It will also be useful to consider the slightly bigger group $\wt{L}$: it has an extra generator $T$ and the relation
$$2T=\ov{p}.$$
It fits into an exact sequence
$$0\to \Z\rTo{T} \wt{L}\to \Z/d(a)\to 0.$$
It is easy to see that $\wt{L}$ is generated by $T$ and $\tau$ with the defining relation
\begin{equation}\label{dn-tau-eq}
d(a)\tau=(-1)^n2(d(a)-\mu(a))T.
\end{equation}
 
Note that for every $\ell\in \wt{L}$ we have a natural  
grading shift operation for a graded matrix factorization of $p_a$:
$$M\mapsto M(\ell),$$
where $M(T):=M[1]$. 
In addition, we denote
$$M(i):=M(i\tau).$$

Since $\wt{L}$ is generated by $\tau$ and $T$, for every $\ell$, we have $M(\ell)=M(i)[j]$ for some $i,j$.
We also have a functorial isomorphism
$$M(d(a))\simeq M[(-1)^n2(d(a)-\mu(a))].$$

For an $L$-homogeneous ideal $\II\sub \C[x_1,\ldots,x_n]$ such that $p_a\in \II$ we denote by $\stab(\II)$ the graded matrix factorization of $p_a$
corresponding to the module $\OO/\II$.

In particular, we consider the following graded matrix factorization of $p_a$:
$$E:=\begin{cases} \stab(x_2,x_4,\ldots,x_n), & n \text{ even},\\
\stab(x_1,x_3,\ldots,x_n), & n \text{ odd},\end{cases}.$$
Note that by definition the grading of $x_i$ is $\ov{x}_i$.


We denote by $\MF_{\Ga}(p_a)$ the dg-category of graded matrix factorizations of $p_a$. We denote by
$\Hom^*$ (or $\Ext^*$) the cohomology of the morphism complexes in this category.
For most of our considerations it will be enough to do computations on the level of cohomology (however,
we will use existence of various natural functors as quasi-functors at the dg-level).

By the main result of \cite{Aramaki}, for any $i\in\Z$, the collection 
$$(E(i),E(i+1),\ldots,E(i+\mu^\vee(a)-1))$$
is a full exceptional collection in $\MF_{\Ga}(p_a)$. We refer to it (for $i=0$) as the {\it AT exceptional collection}.
We should point out that in the original proof of \cite{Aramaki} there are gaps in the proofs of Lemmas 4.7 and 4.10.
These can be filled using the results of Hirano-Ouchi in \cite[Sec.\ 4.2]{HO} (especially 
\cite[Lem.\ 4.4, Lem.\ 4.5]{HO}), where the fully faithful embedding
needed for the induction is constructed using VGIT technique (note that these are different VGIT embeddings than
the ones used in Sec.\ \ref{VGIT} below).

We denote by $AT(a)$ the directed $A_\infty$-category corresponding to the AT exceptional collection in $\MF_{\Ga}(p_a)$.

\subsubsection{$\Ext$-algebra}\label{Ext-sec}

Let us consider the associative algebra
$$\BB_a:=\bigoplus_{\ell\in\wt{L}}\Hom^0(E,E(\ell)),$$
where $a=(a_1,\ldots,a_n)$.
Then by \cite[Lem.\ 4.1, Lem.\ 4.2]{Aramaki} (extended to the case $a_1=2$), one has an isomorphism of $\wt{L}$-graded algebras
\begin{equation}\label{AT-Ext-algebra-eq}
\BB_a\simeq \begin{cases}\k[x_1,x_3,\ldots,x_{n-1}]/(x_1^{a_1},\ldots,x_{n-1}^{a_{n-1}}), & n \text{ even},\\
\k[x_0,x_2,\ldots,x_{n-1}]/(x_0^2-\varepsilon x_2,x_2^{a_2},\ldots,x_{n-1}^{a_{n-1}}), & n \text{ odd},\end{cases}
\end{equation}
where 
$$\varepsilon=\begin{cases} 0, & a_1>2,\\ 1, & a_1=2.\end{cases}$$
The $\wt{L}$-gradings of $x_i$ are given as follows:
$$\deg(x_0)=\tau+T$$
and for $i>0$,
\begin{equation}\label{deg-xi-eq}
\deg(x_i)=\ov{x}_i=(-1)^{i-1}d(a_1,\ldots,a_{i-1})\ov{x}_1+(-1)^i2(d(a_1,\ldots,a_{i-1})-\mu(a_1,\ldots,a_{i-1}))T.
\end{equation}

Note that $\BB_a$ has a natural monomial basis, and the elements of this basis have distinct degrees in
$\wt{L}/\Z\cdot T\simeq \Z/d(a)$. This implies that whenever $0\le j-i<d(a)$, the space 
$\Ext^*(E(i),E(j))$ is at most $1$-dimensional, and can be identified with the graded component of degree $(j-i)\tau$ in $\BB_a$,
with appropriate shift.

Also, we see that the algebra $\BB_a$ is Gorenstein with the $1$-dimensional socle in degree
$\mu^\vee(a-)\mod d(a)$. This implies that for $\mu^\vee(a-)<j-i<d(a)$ one has $\Ext^*(E(i),E(j))=0$, while
$\Ext^*(E,E(\mu^\vee(a-))$ is $1$-dimensional and
the compositions
$$\Ext^*(E(i),E(\mu^\vee(a-))\ot \Ext^*(E,E(i))\to \Ext^*(E,E(\mu^\vee(a-))$$
are perfect pairing. The latter property will play a crucial role below.






\subsubsection{Serre functor on the category of matrix factorizations}

By \cite[Prop.\ 2.9]{Aramaki}, the Serre functor on $\MF_{\Ga}(p_a)$ is given by $M\mapsto M(\ell_{\SS})$, where
$$\ell_{\SS}=nT-\ov{x}_1-\ldots-\ov{x}_n.$$
Combining this with \eqref{deg-xi-eq} and taking into account \eqref{dn-tau-eq}, we get the following formula.

\begin{lem}\label{Serre-Bside-lem}
The Serre functor on $\MF_{\Ga}(p_a)$ is given by $M\mapsto M(\ell_{\SS})$, with
$$\ell_{\SS}=-\mu^\vee(a)\tau+(n+2m(a))T=\mu^\vee(a-)\tau+(n+2m(a-))T,$$
where
\begin{equation}\label{Serre-shift-ma-eq}
m(a)=(-1)^n\mu^\vee(a)-1+\mu(a_1)-\mu(a_1,a_2)+\ldots+(-1)^{n-1}\mu(a_1,\ldots,a_n).
\end{equation}
\end{lem}

It follows that up to shifts, the helix generated by the AT exceptional collection is simply $(E(i))_{i\in\Z}$.

\subsection{VGIT embedding}\label{VGIT}


Here we record a specialization of the construction in \cite{FKK}, which itself  is a particular case of the general
VGIT construction in \cite{BFK}.

Let us consider the polynomials
$$W=x_1^{a_1}x_2+x_2^{a_2}x_3+\ldots+x_n^{a_n}x_{n+1}^{a_n},$$
$$w_+=x_1^{a_1}x_2+x_2^{a_2}x_3+\ldots+ x_n^{a_n},$$
$$w_-=x_1^{a_1}x_2+x_2^{a_2}x_3+\ldots+x_{n-1}^{a_{n-1}}+x_{n+1}^{a_n}.$$

Note that $W$ is invariant with respect to the $\G_m$-action on $\A^{n+1}$ with the following weights:
\begin{itemize}
\item $c_1=1, \ c_2=-a_1, \ c_3=a_1a_2, \ldots, \ c_n=-a_1a_2\ldots a_{n-1}, \ c_{n+1}=a_1a_2\ldots a_{n-1}$, for $n$ even;
\item $c_1=-1, \ c_2=a_1, \ c_3=-a_1a_2, \ldots, \ c_n=-a_1a_2\ldots a_{n-1}, \ c_{n+1}=a_1a_2\ldots a_{n-1}$, for $n$ odd.
\end{itemize}
The main idea of \cite{FKK} is to apply the VGIT construction to this $\G_m$-action.

Let us set
$$\a_n:=a_1a_2\ldots a_n+a_1a_2\ldots a_{n-2}+\ldots,$$
where the last term is $1$ if $n$ is even, and $a_1$ if $n$ is odd.
Note that $\mu^\vee(a)=\a_n-\a_{n-1}$.

We define the intervals of weights as follows:
$$I^-=[0,\a_{n-1}-1]\sub I^+=[0,a_1a_2\ldots a_{n-1}+\a_{n-2}-1].$$
We consider the corresponding windows
$$\WW_{I^-}\sub \WW_{I^+}\sub \MF_{\Ga}(W).$$
Here we use the embedding of $\G_m$ into $\Ga$,
$$\la: t\mapsto (t^{-c_1},t^{-c_2},\ldots,t^{-c_{n+1}}),$$
and consider weights of the restriction of a matrix factorization to the origin, so for example,
the weight of $\k(x_i)$ is $\mu(x_i,\la)=-c_i$. 

We have natural restruction functors
$$r_+:\MF_{\Ga}(W)\to \MF_{\Ga_+}(w_+): E\mapsto E|_{x_{n+1}=1},$$
$$r_-:\MF_{\Ga}(W)\to \MF_{\Ga_-}(w_-): E\mapsto E|_{x_n=1},$$
Here $\Ga$ is the group of diagonal transformations preserving $W$ up to rescaling;
$\Ga_\pm$ are similar groups for $w_\pm$.

\begin{thm} (\cite{FKK})
The functors
$$r_\pm|_{\WW_{I^\pm}}:\WW_{I^\pm}\to \MF_{\Ga_\pm}(w_\pm)$$
are equivalences.
Hence, there exists a fully faithful functor $\Phi$ making the following diagram  
commutative:
\begin{diagram}
\WW_{I^-}&\rTo{}&\WW_{I^+}\\
\dTo_{\sim}^{r_-}&&\dTo_{\sim}^{r_+}\\
\MF_{\Ga_-}(w_-)&\rTo{\Phi}&\MF_{\Ga_+}(w_+)
\end{diagram}
\end{thm}

\begin{proof}
Since our result is a bit more precise than that of \cite{FKK}, we will give the proof. Consider the ideals
$$\II_+:=(x_j \ |\ c_j>0), \ \ \II_-=(x_j \ |\ c_j<0)$$
in $\k[x_1,\ldots,x_{n+1}]$, and let us set
$$Y_\pm=\A^{n+1}\setminus Z(\II_\pm), \ \ U_+=\A^{n+1}\setminus Z(x_{n+1}), \ \
U_-=\A^{n+1}\setminus Z(x_n).$$
Then \cite[Lem.\ 3.7]{FKK} states that the natural restriction functors
$$\MF_{\Ga}(Y_\pm,W)\to \MF_{\Ga}(U_\pm,W)$$
are equivalences.

On the other hand, it is easy to see that the restriction functors
$$\MF_{\Ga}(U_+,W)\to \MF_{\Ga_+}(w_+): E\mapsto E|_{x_{n+1}=1},$$
$$\MF_{\Ga}(U_-,W)\to \MF_{\Ga_-}(w_-): E\mapsto E|_{x_n=1}$$
are equivalences (see \cite[Lem.\ 2.3]{FKK}).

Finally, we claim that \cite[Cor.\ 3.2.2+Prop.\ 3.3.2]{BFK} imply that the compositions
$$\WW_{I^\pm}\hra \MF_{\Ga}(W)\to \MF_{\Ga}(Y_{\pm},W)$$
are equivalences. Indeed, we observe that
$$Z(\II_\pm)=\{x\in \A^{n+1} \ | \lim_{t\to 0} \la^\pm(t)x=0\}.$$
The lengths of the intervals $d^\pm$ giving the windows are given by
$$d^\pm=-\sum_{i: x_i\in \II_\pm}\mu(x_i,\la^\pm)-1=\pm\sum_{i:x_i\in \II_\pm} c_i -1.$$
(see \cite[Sec.\ 3.1]{BFK}). Thus, we get
$$d^+=[c_{n+1}+c_{n-1}+\ldots]-1=a_1a_2\ldots a_{n-1}+\a_{n-2}-1,$$
$$d^-=-[c_n+c_{n-2}+\ldots]-1=\a_{n-1}-1.$$
\end{proof}

Note that $w_+=p_a$, whereas 
$$w_-=p_{a-}+x_{n+1}^{a_n}.$$
where we set $a-:=(a_1,\ldots,a_{n-1})$.
We combine the above functor with the embedding
$$\iota:\MF_{\Ga_{a-}}(p_{a-})\to \MF_{\Ga_-}(w_-): F\mapsto F\boxtimes \stab(x_{n+1}).$$
This allows us to define the fully faithful functor
$$\Phi_0:=\Phi\circ\iota: \MF_{\Ga_{a-}}(p_{a-})\to \MF_{\Ga_a}(p_a).$$


\begin{lem}\label{image-functor-lem}
(i) Assume $n$ is even. Then
$$\stab(x_1,x_3,\ldots,x_{n-1},x_nx_{n+1})(-i\ov{x}_1)\in \WW_{I^-} \ \text{ for } 0\le i\le \mu^\vee(a-)-1,$$
$$r_-(\stab(x_1,x_3,\ldots,x_{n-1},x_nx_{n+1})(-i\ov{x}_1))\simeq \stab(x_1,x_3,\ldots,x_{n-1},x_{n+1})(-i\ov{x}_1),$$
$$r_+(\stab(x_1,x_3,\ldots,x_{n-1},x_nx_{n+1})(-i\ov{x}_1))\simeq \stab(x_1,x_3,\ldots,x_{n-1},x_n)(-i\ov{x}_1),$$
$$\Phi_0(E(i))=\stab(x_1,x_3,\ldots,x_{n-1},x_n)(-i\ov{x}_1) \ \text{ for } 0\le i\le \mu^\vee(a-)-1.$$

\noindent
(ii) Assume $n$ is odd. Then
$$\stab(x_2,x_4,\ldots,x_{n-1},x_nx_{n+1})(i\ov{x}_1)\in \WW_{I^-} \ \text{ for } 0\le i\le \mu^\vee(a-)-1,$$
$$r_-(\stab(x_2,x_4,\ldots,x_{n-1},x_nx_{n+1})(i\ov{x}_1))\simeq \stab(x_2,x_4,\ldots,x_{n-1},x_{n+1})(i\ov{x}_1),$$
$$r_+(\stab(x_2,x_4,\ldots,x_{n-1},x_nx_{n+1})(i\ov{x}_1))\simeq \stab(x_2,x_4,\ldots,x_{n-1},x_n)(i\ov{x}_1),$$
$$\Phi_0(E(i))=\stab(x_2,x_4,\ldots,x_{n-1},x_n)(i\ov{x}_1) \ \text{ for } 0\le i\le \mu^\vee(a-)-1.$$
\end{lem}

\begin{proof}
(i) We have 
\begin{align*}
&-\mu(x_1,\la)=c_1=1, -\mu(x_3,\la)=c_3=a_1a_2, \ldots, -\mu(x_{n-1},\la)=c_{n-1}=a_1a_2\ldots a_{n-2},\\ 
&\mu(x_nx_{n+1},\la)=0.
\end{align*}
Hence, the $\la$-weights of $\stab(x_1,x_3,\ldots,x_{n-1},x_nx_{n+1})|_0$ are given by the weights of the elements
of the exterior algebra with generators of weights $c_1,\ldots,c_{n-1},0$, so they lie in
the interval 
$$[0,c_1+c_3+\ldots+c_{n-1}]=[0,\a_{n-2}].$$
Thus, for $0\le i\le \mu^\vee(a-)-1$,
the weights of $\stab(x_1,x_3,\ldots,x_{n-1},x_nx_{n+1})(-i\ov{x}_1)|_0$ will lie in the segment
from $0$ to
$$\a_{n-2}+\mu^\vee(a-)-1=\a_{n-1}-1.$$

\noindent
(ii) The proof is completely analogous to (i), using the weights of $x_2,x_4,\ldots,x_{n-1},x_nx_{n+1}$.
\end{proof}

\subsection{Dual exceptional collections}\label{dual-exc-coll-sec}

Recall that
$$E=\begin{cases} \stab(x_2,x_4,\ldots,x_n), & n \text{ even},\\
\stab(x_1,x_3,\ldots,x_n), & n \text{ odd},\end{cases}$$
Let us consider another graded matrix factorization of $p_a$:
$$F:=\begin{cases} \stab(x_1,x_3,\ldots,x_{n-1},x_n), & n \text{ even},\\
\stab(x_2,x_4,\ldots,x_{n-1},x_n), & n \text{ odd}.\end{cases}$$

\begin{lem}\label{EF-Hom-lem}
(i) One has an ungraded isomorphism
$$\Hom^*(E,F(i))=\begin{cases} \k, & i\equiv \a_{n-3},\a_{n-1} \mod(a_1a_2\ldots a_n),\\
0, & \text{otherwise}.\end{cases}.$$
The degrees are determined as follows: we have
$$\Hom^{\frac{n}{2}-1}(E,F(-\ov{x}_2-\ov{x}_4-\ldots-\ov{x}_{n-2}))=
\Hom^{\frac{n}{2}}(E,F(-\ov{x}_2-\ov{x}_4-\ldots-\ov{x}_n))=\k \ \text{ if } n \ \text{ is even},$$
$$\Hom^{\frac{n-1}{2}}(E,F(-\ov{x}_1-\ov{x}_3-\ldots-\ov{x}_{n-2}))=
\Hom^{\frac{n+1}{2}}(E,F(-\ov{x}_1-\ov{x}_3-\ldots-\ov{x}_n))=\k \ \text{ if } n \ \text{ is odd}.$$

\noindent
(ii) One has an ungraded isomorphism
$$\Hom^*(F,E(i))=\begin{cases} \k, & i\equiv -\a_{n-2},a_1a_2\ldots a_{n-1}-\a_{n-2} \mod(a_1a_2\ldots a_n),\\
0, & \text{otherwise}.\end{cases}$$
The degrees are determined as follows: we have
$$\Hom^{\frac{n}{2}}(F,E(-\ov{x}_1-\ov{x}_3-\ldots-\ov{x}_{n-1}))=
\Hom^{\frac{n}{2}+1}(F,E(-\ov{x}_1-\ov{x}_3-\ldots-\ov{x}_{n-1}-\ov{x}_n))=\k \ \text{ if } n \ \text{ is even},$$
$$\Hom^{\frac{n-1}{2}}(F,E(-\ov{x}_2-\ov{x}_4-\ldots-\ov{x}_{n-1}))=
\Hom^{\frac{n+1}{2}}(F,E(-\ov{x}_2-\ov{x}_4-\ldots-\ov{x}_{n-1}-\ov{x}_n))=\k \ \text{ if } n \ \text{ is odd}.$$
\end{lem}

\begin{proof}
This is a standard computation based on the quasiisomoprhism 
$$\Hom(E,\stab(a_1,\ldots,a_k))\simeq E^\vee|_{a_1=\ldots=a_k}$$
for a regular sequence $a_1,\ldots,a_k$ (see e.g. \cite[Lem.\ 4.2]{Dyckerhoff}).
\end{proof}

\begin{cor}\label{EF-Hom-cor} 
Let us define the integer $N(n)$ by the following relation in $L$:
$$-\ov{x}_2-\ov{x}_4-\ldots-\ov{x}_{n-2}=\a_{n-3}\tau+N(n)\cdot \ov{p}, \ \text{ if } n \ \text{ is even},$$
$$-\ov{x}_1-\ov{x}_3-\ldots-\ov{x}_{n-2}=\a_{n-3}\tau+N(n)\cdot \ov{p}, \ \text{ if } n \ \text{ is odd}.$$
Then
$$\Hom^{\lfloor \frac{n-1}{2} \rfloor +2N(n)}(E,F(\a_{n-3}))=\k.$$
\end{cor}

\begin{prop}\label{MF-dual-prop}
Let us consider the subcategory
\begin{equation}\label{B-subcat-defi}
\BB=\langle E(-\a_{n-2}), E(1-\a_{n-2}),\ldots, E(-\a_{n-3}-1)\rangle.
\end{equation}
Let 
$$L_\BB:{}^{\perp}\BB\to \BB^{\perp}$$ 
denote the left mutation functor (which is an equivalence).
Then the exceptional collection
\begin{equation}\label{F-col-eq}
(F(\mu^\vee(a-)-1),\ldots,F(1),F)[\lfloor \frac{n-1}{2} \rfloor +2N(n)],
\end{equation}
where $N(n)$ is defined in Corollary \ref{EF-Hom-cor},
is left dual to the exceptional collection
\begin{equation}\label{L-B-E-col-eq}
L_\BB(E(-\a_{n-3})),L_\BB(E(-\a_{n-3}+1)),\ldots,L_\BB(E(\mu^\vee(a-)-\a_{n-3}-1)).
\end{equation}
\end{prop}

\Pf . To begin with, by Lemma \ref{EF-Hom-lem}, the only nonzero morphisms from objects of the collection
$$E(-\a_{n-3}),E(-\a_{n-3}+1),\ldots,E(\mu^\vee(a-)-\a_{n-3}-1)$$
to objects of the collection \eqref{F-col-eq} are of the form
$$\Hom^*(E(-\a_{n-3}+i),F(i))=\k, \text{ for } i=0,\ldots,\mu^\vee(a-)-1.$$
Also, by Lemma \ref{EF-Hom-lem}, the collection \eqref{F-col-eq} belongs to $\BB^{\perp}$.
It follows that the only nonzero morphisms from objects of the collection \eqref{L-B-E-col-eq}
to those of \eqref{F-col-eq} are
$$\Hom^*(L_\BB(E(-\a_{n-3}+i)),F(i))=\k, \text{ for } i=0,\ldots,\mu^\vee(a-)-1.$$

Set
\begin{equation}\label{C-subcoll-eq}
\CC=\langle E(-\a_{n-3}),E(-\a_{n-3}+1),\ldots,E(\mu^\vee(a-)-\a_{n-3}-1)\rangle,
\end{equation}
and let $\CC'$ denote the subcategory generated by the collection \eqref{F-col-eq}.
It remains to prove that $\CC'$ is contained in the subcategory generated by the collection \eqref{L-B-E-col-eq},
i.e., $\CC'\sub L_\BB(\CC)$.
To this end, we first observe that we have a semiorthogonal decomposition
$$\MF_{\Ga_a}(p_a)=\langle E(-\a_{n-2}),L_\BB(\CC),\BB,\DD\rangle,$$
where 
$$\DD=\langle E(\mu^\vee(a-)-\a_{n-3}),E(-\mu^\vee(a-)-\a_{n-3}+1),\ldots,E(\mu^\vee(a)-\a_{n-2}-1)\rangle.$$

By Lemma \ref{EF-Hom-lem}, we have 
$$\CC'\sub \langle \BB,\DD\rangle^\perp,$$
so we get an inclusion
$$\CC'\sub\langle E(-\a_{n-2}),L_\BB(\CC).$$
On the other hand, again by Lemma \ref{EF-Hom-lem}, we have
$$\CC'\sub \langle E(-\a_{n-2})\rangle^\perp,$$
so we deduce that $\CC'\sub L_\BB(\CC)$.
\ed

\begin{cor}\label{common-obj-cor}
One has $L_\BB(E(-\a_{n-3}))\simeq F[\lfloor \frac{n-1}{2} \rfloor +2N(n)]$.
\end{cor}

Putting together the above computations we derive the following result.
Let us consider the functor
$$\Psi:\MF_{\Ga_{a-}}(p_{a-})\to \MF_{\Ga_a}(p_a):X\mapsto R_\BB\bigl((\Phi_0X)(\mu^\vee(a-)-1)\bigr)(\a_{n-3})[\lfloor \frac{n-1}{2} \rfloor +2N(n)],$$
where $\BB\sub \MF_{\Ga_a}(p_a)$ is given by \eqref{B-subcat-defi}. 

\begin{thm}\label{emb-dual-thm}
The functor $\Psi$ is fully faithful and 
$$\Psi(E),\Psi(E(1)),\ldots,\Psi(E(\mu^\vee(a-)-1))$$
is the left dual collection to the exceptional collection
$$E, E(1),\ldots, E(\mu^\vee(a-)-1)$$
\end{thm}

\Pf . The computation of Lemma \ref{image-functor-lem} gives
\begin{equation}\label{Phi0-image-eq}
\Phi_0(E(i))\simeq F(-i) \ \text{ for } 0\le i\le \mu^\vee(a-)-1.
\end{equation}
Hence, from Proposition \ref{MF-dual-prop} we get that the image of $X\mapsto (\Phi_0 X)(\mu^\vee(a-)-1)$ is contained in
${}^{\perp}\BB$. Since $R_\BB:{}^{\perp}\BB\to \BB^{\perp}$ is an equivalence,
we derive that $\Psi$ is fully faithful. The duality of the needed collections follows from
\eqref{Phi0-image-eq} and from Proposition \ref{MF-dual-prop}.
\ed

\subsection{Recovering the collection from the initial segment}\label{B-side-recursion-sec}

\subsubsection{Perfect pairing property}

Theorem \ref{emb-dual-thm} implies that the directed $A_\infty$-category corresponding to the subcollection 
$$(E,\ldots,E(\mu^\vee(a-)-1))$$
 of the AT-collection in $\MF_{\Ga}(p_a)$ is equivalent to the directed $A_\infty$-category corresponding to
the right dual of the AT-collection in $\MF(p_{a-})$.
Now we need to identify the relation of the next object $E(\mu^\vee(a-))$ to this subcollection.

For this we use the following general observations about exceptional collections.
Let $E_1,\ldots,E_{m+1}$ be an exceptional collection in a triangulated $A_\infty$-category $\DD$, and consider the subcategory
$$\CC:=\langle E_1,\ldots,E_m\rangle$$ 
Let $\la,\rho:\DD\to \CC$ denote the left and right adjoint functors to the inclusion, and let $\SS_{\CC}$ denote the Serre functor on the subcategory $\CC$.

\begin{lem}\label{perfect-pairing-lem}
The following conditions are equivalent.

\noindent
(i) $\Hom^*(E_1,E_{m+1})=\Hom^d(E_1,E_{m+1})=\k$
and for each $i$, $1<i<m+1$, the compositions
$$\Hom^j(E_i,E_{m+1})\ot\Hom^{d-j}(E_1,E_i)\to \Hom^d(E_1,E_{m+1})=\k,$$
for all $j$ are perfect pairings.

\noindent
(i') $\Hom^*(E_1,E_{m+1})=\Hom^d(E_1,E_{m+1})=\k$
and for each $C\in \CC$, the compositions
$$\Hom^d(C,E_{m+1})\ot\Hom^0(E_1,C)\to \Hom^d(E_1,E_{m+1})=\k,$$
are perfect pairings.

\noindent
(i'') $\Hom^*(E_1,E_{m+1})=\Hom^d(E_1,E_{m+1})=\k$ and for each $i$, $1<i<m+1$, one has
$$\Hom^*(L_{E_1}E_i,E_{m+1})=0.$$

\noindent
(ii) One has an isomorphism 
$$\rho(E_{m+1})[d]\simeq \SS_{\CC}(E_1).$$

\noindent
(ii') One has an isomorphism 
$$\la(L_{\CC}(E_{m+1}))[d-1]\simeq \SS_{\CC}(E_1).$$ 

\noindent
(iii) For any exceptional collection $(E'_1,\ldots,E'_m)$ generating $\CC$, there is an equivalence of directed $A_{\infty}$-categories
$$\eend_{\scriptscriptstyle\rightarrow}(\SS_{\CC}(E_1),E'_1,\ldots,E'_m)\simeq 
\eend_{\scriptscriptstyle\rightarrow} (L_{\CC}(E_m)[d-1],E'_1,\ldots,E'_m),$$
identical on $\eend_{\scriptscriptstyle\rightarrow}(E'_1,\ldots,E'_m)=\eend(E'_1,\ldots,E'_m)$.
\end{lem}

\Pf . (i)$\iff$(i'). The pairing in (i') corresponds to a morphism of cohomological functors
$$\Hom(C,E_{m+1}[d])\to \Hom(E_1,C)^\vee.$$
Condition (i) states that this morphism is an isomorphism for the generators $(E_i[n])_{1\le i\le m}$ of $\CC$.
Hence, the assertion follows from the five-lemma.

\noindent
(i)$\iff$(i'').
For every $E_i$ with $1<i<m+1$, we have an exact triangle
$$L_{E_1}(E_i)[-1]\to R\Hom(E_1,E_i)\ot E_1\to E_i\to L_{E_1}(E_i)$$
Taking $\Hom(?,E_{m+1}[d])$ we get an exact sequence
$$\ldots\to \Hom^{d+j}(L_{E_1}(E_i),E_{m+1})\to \Hom^{d+j}(E_i,E_{m+1})\to \Hom^{-j}(E_1,E_i)^\vee\ot\Hom^d(E_1,E_{m+1})\to\ldots$$
Now we see that the perfect pairing property for $E_i$ is equivalent to the vanishing
$$\Hom^*(L_{E_1}(E_i),E_{m+1})=0.$$

\noindent
(i')$\iff$(ii). 
Condition (i') is equivalent to a functorial isomorphism in $C\in\CC$,
$$\Hom(C,E_{m+1}[d])\simeq \Hom(E_1,C)^\vee.$$
But we have functorial identification 
$$\Hom(C,E_{m+1}[d])\simeq \Hom(C,\rho(E_{m+1})[d]),$$
$$\Hom(E_1,C)^\vee\simeq \Hom(C,\SS_{\CC}(E_1)).$$
Hence, (i') is equivalent to a functorial isomorphism in $C\in\CC$,
$$\Hom(C,\rho(E_{m+1})[d])\simeq\Hom(C,\SS_{\CC}(E_1)),$$
i.e., to an isomorphism $\rho(E_{m+1})[d]\simeq \SS_{\CC}(E_1)$.

\noindent (ii)$\iff$(ii'). This follows from Lemma \ref{adjoints-mutation-lem}.

\noindent (ii')$\iff$(iii). By adjunction, we have
$$\eend_{\scriptscriptstyle\rightarrow} (L_{\CC}(E_{m+1})[d-1],E'_1,\ldots,E'_{m})\simeq
\eend_{\scriptscriptstyle\rightarrow} (\la(L_{\CC}(E_{m+1}))[d-1],E'_1,\ldots,E'_{m}),$$
so condition (iii) simply states that the $A_{\infty}$-modules
corresponding to $\SS_{\CC}(E_1)$ and $\la(L_{\CC}(E_{m+1}))[d-1]$ are equivalent.
It remains to use the fact that the functor $C\mapsto \hom(C,E'_1\oplus\ldots\oplus E'_{m})$
gives an equivalence of $\CC^{op}$ with the category of left $A_\infty$-modules over 
$\eend(E'_1\oplus\ldots\oplus E'_{m})$. 
\ed

We will call the property (i) the {\it perfect pairing property for the collection} $E_1,\ldots,E_{m+1}$.
Note that as we observed in Sec.\ \ref{Ext-sec},
this property holds for the initial segment $(E,E(1),\ldots,E(\mu^\vee(a-)))$ of the AT exceptional collection in
$\MF_{\Gamma_a}(p_a)$.


\subsubsection{Adjoints and mutations}

Assume that we have an exceptional collection $E_1,\ldots,E_{m+l}$ in a triangulated $A_\infty$-category $\DD$, with $0\le l\le m$, such that 

\begin{itemize}
\item there exists an autoequivalence $\tau$ of $\DD$ such that $\tau(E_i)=E_{i+1}$;
\item $\Hom^*(E_i,E_j)=0$ for $j-i>m$.
\end{itemize}

Let $\CC=\langle E_1,\ldots,E_m\rangle$, and let $\rho:\DD\to\CC$ be the right adjoint functor to the inclusion.

\begin{lem}\label{perfect-pairing-rho-lem}
(i) Assume that the perfect pairing property holds for $E_1,\ldots,E_{m+1}$, with $\Hom^d(E_1,E_{m+1})=\k$.
Then for $i=1,\ldots,l$, one has
\begin{equation}\label{rho-S-1-isom-eq}
\rho(E_{m+i})[d]\simeq \SS_{\CC}(E_i).
\end{equation}

\noindent
(ii) Assume in addition that for every pair of morphisms $\a:E_1\to E_i[a]$ and $\b:E_i\to E_{m+1}[d-a]$, such that $1<i\le l$, one has
\begin{equation}\label{tau-commutativity}
\tau^{i-1}(\b\circ\a)=\tau^m(\a)\circ \b.
\end{equation}
Then the restriction of $\rho$,
$$\rho:\langle E_{m+1},\ldots,E_{m+l}\rangle\to \CC$$
is fully faithful. 
\end{lem}

\Pf . (i) Note that for each $i=1,\ldots,l$, the collection $E_i,E_{i+1},\ldots,E_{m+i}$ is the image of $E_1,\ldots,E_{m+1}$ under $\tau^{i-1}$,
hence, the perfect pairing property holds for $E_i,E_{i+1},\ldots,E_{m+i}$.
We claim that this property also holds for the collection 
$$E_i,R_{E_i}E_1,R_{E_i}E_2\ldots,R_{E_i}E_{i-1},E_{i+1},\ldots,E_m,E_{m+i}.$$
Indeed, this follows immediately from Lemma \ref{perfect-pairing-lem} since
$$\Hom^*(L_{E_i}R_{E_i}E_j,E_{m+i})=\Hom^*(E_j,E_{m+i})=0$$
for $j\le i-1$. 

Since $\langle E_i,R_{E_i}E_1,\ldots,R_{E_i}E_{i-1},E_{i+1},\ldots,E_m\rangle=\CC$, by Lemma \ref{perfect-pairing-lem},
we deduce an isomorphism
$$\rho(E_{m+i})[d]\simeq \SS_{\CC}(E_i).$$

\noindent
(ii) Equation \eqref{tau-commutativity} implies that a similar property holds for any pair $\a:E_i\to E_j[a]$ and $\b:E_j\to E_{m+i}[d-a]$, where $i<j\le l$.
Let us choose identifications $\Hom^d(E_i,E_{m+i})\simeq \k$ for all $i$, compatibly with $\tau$.
Then the above property implies that for any object $C\in \CC$ and any morphism $\a:E_i\to E_j[a]$, the
following diagram is commutative:
\begin{diagram}
\Hom^d(C,E_{m+i})\ot\Hom^{-a}(E_j,C)&\rTo{\id\ot (?\circ\a)}&\Hom(C,E_{m+i}[d])\ot\Hom(E_i,C)\\
\dTo{(\tau^m(\a)\circ?)\ot \id}&&\dTo{}\\
\Hom^{d+a}(C,E_{m+j})\ot\Hom^{-a}(E_j,C)&\rTo{}& \k
\end{diagram}
Indeed, for $f\in \Hom^d(C,E_{m+i})$ and $g\in \Hom^{-a}(E_j,C)$, we have equality
$$(f\circ g)\circ\a=\tau^m(\a)\circ (f\circ g).$$
Equivalently, the following diagram is commutative for any $C\in \CC$:
\begin{diagram}
\Hom(E_i,C)^\vee &\rTo{}& \Hom(C,E_{m+i}[d])\\
\dTo{(?\circ\a)^\vee}&&\dTo{\tau^m(\a)\circ?}\\
\Hom(E_j,C)^\vee &\rTo{}& \Hom(C,E_{m+j}[d])
\end{diagram}
which leads to the commutative diagram
\begin{diagram}
\SS_{\CC}(E_i)&\rTo{}& \rho(E_{m+i})[d]\\
\dTo{\SS_{\CC}(\a)}&&\dTo{\rho\tau^m[d](\a)}\\
\SS_{\CC}(E_j)[a]&\rTo{}& \rho(E_{m+j})[d+a]
\end{diagram}
for every $\a:E_i\to E_j[a]$.
Since the horizontal arrows are isomorphisms (see Lemma \ref{perfect-pairing-lem}),
It follows that the composed map 
$$\Hom^*(E_i,E_j)\rTo{\tau^m}\Hom^*(E_{m+i},E_{m+j})\to \Hom^*(\rho(E_{m+i}),\rho(E_{m+j}))$$
gets identified with $\a\mapsto \SS_{\CC}(\a)$, so it 
is an isomorphism. Hence, the restriction of $\rho$ to $\lan E_{m+1},\ldots,E_{m+j}\ran$ is fully faithful.
\ed

\begin{lem}\label{fully-faithful-lem}
Let $\DD=\langle \CC_0,\CC_1,\ldots,\CC_n\rangle$
be a semiorthogonal decomposition, and let $\rho_i:\DD\to \CC_i$ denote the right adjoint functor to the inclusion.
Assume that 
\begin{itemize}
\item $\Hom(\CC_i,\CC_j)=0$ for $j>i+1$;
\item for every $i<n$, the restriction
$$\rho_i|_{\CC_{i+1}}:\CC_{i+1}\to \CC_i$$
is fully faithful. 
\end{itemize}
Then we have canonical isomorphisms of functors
\begin{equation}\label{iterated-rho-isomorphism}
\rho_0 L_{\CC_1}\ldots L_{\CC_{i-1}}|_{\CC_i}\simeq \rho_0\rho_1\ldots\rho_{i-1}|_{\CC_i}[i-1],
\end{equation} 
and for every $i\ge j\ge 1$, for $C_i\in \CC_i$, $C_j\in \CC_j$,
the functor $\rho_0$ gives an isomorphism 
$$\Hom(L_{\CC_1}\ldots L_{\CC_{i-1}}C_i,L_{\CC_1}\ldots L_{\CC_{j-1}}\CC_j)\rTo{\sim} \Hom(\rho_0 L_{\CC_1}\ldots L_{\CC_{i-1}}\CC_i,\rho_0 L_{\CC_1}\ldots L_{\CC_{j-1}}\CC_j).$$
In particular, $\rho_0$ is fully faithful on each subcategory $L_{\CC_1}\ldots L_{\CC_{i-1}}\CC_i$.
\end{lem}

\Pf .
{\bf Step 1.} We claim that for any $C_i\in\CC_i$, where $i\ge 1$, and any $C_1\in\CC_1$, the map induced by $\rho_0$,
$$\Hom(L_{\CC_1}\ldots L_{\CC_{i-1}}C_i,C_1)\to \Hom(\rho_0 L_{\CC_1}\ldots L_{\CC_{i-1}}C_i,\rho_0 C_1)$$
is an isomorphism. For $i=1$ this true by assumption, so we can assume $i>1$.
Equivalently, we need to check that the canonical morphism
$$\rho_0 L_{\CC_1}\ldots L_{\CC_{i-1}}C_i\to L_{\CC_1}\ldots L_{\CC_{i-1}}C_i$$
induces an isomorphism on $\Hom(?,C_1)$. Let us consider the commutative square induced by the adjunction morphism for $\rho_0$,
\begin{equation}\label{rho0-rho1-diagram}
\begin{diagram}
\rho_0 L_{\CC_1}\ldots L_{\CC_{i-1}}C_i &\rTo{\sim}& \rho_0 \rho_1 L_{\CC_2}\ldots L_{\CC_{i-1}}C_i[1]\\
\dTo{}&&\dTo{}\\
L_{\CC_1}\ldots L_{\CC_{i-1}}C_i &\rTo{}& \rho_1 L_{\CC_2}\ldots L_{\CC_{i-1}}C_i[1]
\end{diagram}
\end{equation}
Note that the cocone of the bottom horizontal arrow is $L_{\CC_2}\ldots L_{\CC_{i-1}}C_i$ which is in $\CC_0^{\perp}=\ker(\rho_0)$,
so the top horizontal arrow is an isomorphism. Let us consider the induced commutative square
\begin{diagram}
\Hom(\rho_0 L_{\CC_1}\ldots L_{\CC_{i-1}}C_i,C_1) &\lTo{\sim}& \Hom(\rho_0 \rho_1 L_{\CC_2}\ldots L_{\CC_{i-1}}C_i[1],C_1)\\
\uTo{}&&\uTo{}\\
\Hom(L_{\CC_1}\ldots L_{\CC_{i-1}}C_i,C_1)&\lTo{}& \Hom(\rho_1 L_{\CC_2}\ldots L_{\CC_{i-1}}C_i[1],C_1)
\end{diagram}
Note that the right vertical arrow is an isomorphism since $\rho_0|_{\CC_1}$ is fully faithful (we apply this to the objects 
$\rho_1 L_{\CC_2}\ldots L_{\CC_{i-1}}C_i[1]$ and $C_1$ of $\CC_1$). Finally the bottom horizontal arrow is an isomorphism
since $\Hom(L_{\CC_2}\ldots L_{\CC_{i-1}}C_i,C_1)=0$. This implies that the left vertical arrow is an isomorphism as claimed.

Also, applying the isomorphism in diagram \eqref{rho0-rho1-diagram} to the categories $(\CC_1,\ldots,\CC_i)$ we get the functorial isomoprhisms
$$\rho_0 L_{\CC_1}\ldots L_{\CC_{i-1}}C_i\simeq \rho_0 \rho_1 L_{\CC_2}\ldots L_{\CC_{i-1}}C_i[1]\simeq \rho_0\rho_1\rho_2 L_{\CC_3}\ldots L_{\CC_{i-1}}C_i[2].$$
Continuing in this way we derive \eqref{iterated-rho-isomorphism}.

\noindent
{\bf Step 2.} Now we restate the result of Step 1 as 
$$\Hom(L_{\CC_0}L_{\CC_1}\ldots L_{\CC_{i-1}}\CC_i,\CC_1)=0$$
for $i\ge 1$. Indeed, this immediately follows from the exact triangle
$$\rho_0 L_{\CC_1}\ldots L_{\CC_{i-1}}C_i\to L_{\CC_1}\ldots L_{\CC_{i-1}}C_i\to L_{\CC_0}L_{\CC_1}\ldots L_{\CC_{i-1}}\CC_i\to\ldots$$
Similarly, for $i\ge j\ge 1$, we have
$$\Hom(L_{\CC_{j-1}}L_{\CC_j}\ldots L_{\CC_{i-1}}\CC_i,\CC_j)=0.$$

\noindent
{\bf Step 3.} We claim that for any $i\ge j\ge k\ge 1$, one has
$$\Hom(L_{\CC_{k-1}}L_{\CC_k}\ldots L_{\CC_{i-1}}\CC_i,L_{\CC_k}\ldots L_{\CC_{j-1}}\CC_j)=0,$$
or equivalently, for $C_i\in \CC_i$ and $C_j\in \CC_j$, the natural map
$$\Hom(L_{\CC_k}\ldots L_{\CC_{i-1}}C_i,L_{\CC_k}\ldots L_{\CC_{j-1}}\CC_j)\to \Hom(\rho_{k-1} L_{\CC_k}\ldots L_{\CC_{i-1}}\CC_i,\rho_{k-1} L_{\CC_k}\ldots L_{\CC_{j-1}}\CC_j)$$
is an isomorphism.

We use induction on $j-k$. The case $j=k$ is exactly Step 2, so let us assume that $j>k$.
Note that by Step 2, we have
$$\Hom(L_{\CC_{k-1}}L_{\CC_k}\ldots L_{\CC_{i-1}}\CC_i,\CC_k)=0,$$
hence, for $C_i\in \CC_i$ and $C_j\in \CC_j$ we have an isomorphism
$$\Hom(L_{\CC_{k-1}}L_{\CC_k}\ldots L_{\CC_{i-1}}C_i,L_{\CC_k}\ldots L_{\CC_{j-1}}C_j)\simeq
\Hom(L_{\CC_{k-1}}L_{\CC_k}\ldots L_{\CC_{i-1}}C_i,L_{\CC_{k+1}}\ldots L_{\CC_{j-1}}C_j).$$
Since $\Hom(\CC_{k-1},L_{\CC_{k+1}}\ldots L_{\CC_{j-1}}C_j)$, we further have an isomorphism
$$\Hom(L_{\CC_{k-1}}L_{\CC_k}\ldots L_{\CC_{i-1}}C_i,L_{\CC_{k+1}}\ldots L_{\CC_{j-1}}C_j)\simeq
\Hom(L_{\CC_k}\ldots L_{\CC_{i-1}}C_i,L_{\CC_{k+1}}\ldots L_{\CC_{j-1}}C_j)$$
which vanishes by the induction assumption.

Finally, taking $k=1$ we obtain the assertion we wanted to prove.
\ed

\begin{rmk}
Note that the restriction of $\rho_0$ to the subcategory
$$\langle \CC_1,\ldots,\CC_n\rangle=\langle L_{\CC_1}\ldots L_{\CC_{n-1}}\CC_n,\ldots, L_{\CC_1}\CC_2,\CC_1\rangle$$
is not fully faithful provided $\CC_2\neq 0$. Indeed, this is clear since $\rho_0(\CC_2)=0$. Lemma \ref{fully-faithful-lem} only
checks that morphisms from left to right with respect to the mutated semiorthogonal decomposition are preserved. However, we have $\Hom(\CC_1,L_{\CC_1}\CC_2)=0$, 
whereas $\Hom(\rho_0(C_1),\rho_0(L_{\CC_1}C_2))$ is not necessarily zero for $C_1\in \CC_1$, $C_2\in \CC_2$.
\end{rmk}

\begin{prop}\label{recursion-Bside-prop}
Assume that we have an exceptional collection $E_0,\ldots,E_N$ in a triangulated $A_{\infty}$-category $\DD$, and for some $m<N$ the following conditions hold
\begin{itemize}
\item there exists an autoequivalence $\tau$ such that $\tau(E_i)=E_{i+1}$;
\item $\Hom(E_i,E_j)=0$ for $j-i>m$;
\item the perfect pairing property holds for $E_0,\ldots,E_m$ with $\Hom^d(E_0,E_m)=\k$;
\item for every pair of morphisms $\a:E_0\to E_i[a]$ and $\b:E_i\to E_{m}[d-a]$, such that $0<m+i\le N$, equation \eqref{tau-commutativity} holds.
\end{itemize}
Let $F_{-N},\ldots,F_{-1},F_0$ be the left dual exceptional collection to $E_0,\ldots,E_N$, so that
$F_{-m+1},\ldots,F_0$ is the left dual collection to $E_0,\ldots,E_{m-1}$. 
Let $\CC:=\langle E_0,\ldots,E_{m-1}\rangle$ and let $\la:\DD\to\CC$ denote the left adjoint functor to the inclusion (which exists
as an $A_\infty$-functor).
Then 
$$\la(F_{-N})[\lfloor \frac{N}{m}\rfloor(d-1)],\ldots,\la(F_{-i})[\lfloor\frac{i}{m}\rfloor(d-1)],\ldots,\la(F_{-m})[d-1],F_{-m+1},\ldots,F_0$$
is a part of the helix associated with the full exceptional collection $F_{-m+1},\ldots,F_0$ in $\CC$, 
and $\la$ induces an equivalence of directed $A_\infty$-endomorphism algebras 
$$\eend_{\scriptscriptstyle\rightarrow}(F_{-N},\ldots,F_0)\rTo{\sim}\eend_{\scriptscriptstyle\rightarrow}(\la(F_{-N}),\ldots,\la(F_{-m}),F_{-m+1},\ldots,F_0).$$
\end{prop}

\Pf . Let $N=mN_0+r$, where $0\le r<m$, and let us set 
$$\CC_i:=\tau^{mi}(\CC)=\begin{cases}\langle E_{mi},\ldots,E_{mi+m-1}\rangle, & 0\le i<N_0, \\
\langle E_{mN_0},\ldots,E_{mN_0+r} \rangle, & i=N_0.\end{cases}$$
Note that $\CC=\CC_0$ and $\CC_i\sub\CC^{\perp}$ for $i>1$. 
Let also $\rho_i$ denote the right adjoint functor to the inclusion of $\CC_i$.

\noindent
First, we observe that by Lemma \ref{perfect-pairing-rho-lem}, the functor $\rho_0|_{\CC_1}$ is fully faithful and
$$\rho_0(E_{m+j})\simeq \SS_{\CC}(E_j)[-d]$$
for $j=0,\ldots,m-1$.
Using the autoequivalence $\tau$, we deduce that for each $i\le N_0$, the functor $\rho_i|_{\CC_{i+1}}$ is fully faithful 
and
\begin{equation}\label{rho-i-E-Serre-formula}
\rho_i(E_{m(i+1)+j})\simeq \SS_{\CC_i}(E_{mi+j})[-d].
\end{equation}
It follows that for $i<N_0$, the functor $\rho_{i-1}|_{\CC_i}:\CC_i\to \CC_{i-1}$ is an equivalence.


Thus, the conditions of Lemma \ref{fully-faithful-lem} are satisfied for our collection of categories $(\CC_i)$.
Hence, the functor $\rho_0$ preserves morphisms from left to right on the semiorthogonal subcategories 
$$L_{\CC_1}\ldots L_{\CC_{i-1}}\CC_i,\ldots,L_{\CC_1}\CC_2,\CC_1$$
and is fully faithful on each of them.
By Lemma \ref{adjoints-mutation-lem}, this is equivalent to the fact that the functor $\la$ preserves morphisms from left to right on
$$L_{\CC_0}L_{\CC_1}\ldots L_{\CC_{i-1}}\CC_i,\ldots, L_{\CC_0},L_{\CC_1}\CC_2,L_{\CC_0}\CC_1,\CC_0$$
and is fully faithful on each of these subcategories. 

In addition, using \eqref{iterated-rho-isomorphism} and \eqref{rho-i-E-Serre-formula} we compute 
$$\rho_0 L_{\CC_1}\ldots L_{\CC_{i-1}}(E_{mi+j})\simeq \rho_0\rho_1\ldots\rho_{i-1}(E_{mi+j})[i-1]\simeq \SS_{\CC}^i(E_j)[-di+i-1]$$
(we also used the fact that the equivalences $\rho_{i-1}|_{\CC_i}$ for $i<N_0$ commute with the Serre functors). 
Using Lemma \ref{adjoints-mutation-lem} we can rewrite this as
\begin{equation}\label{la-S-isom-eq}
\la L_{\CC_0}L_{\CC_1}\ldots L_{\CC_{i-1}}(E_{mi+j})\simeq \SS_{\CC}^i(E_j)[-di+i].
\end{equation}



By Lemma \ref{block-dual-coll-lem}, the left dual exceptional collection to $E_0,\ldots,E_N$ has form
$$L_{\CC_0}\ldots L_{\CC_{N_0-1}}\tau^{mN_0}(F_{-r+1},\ldots,F_0),L_{\CC_0}\ldots L_{\CC_{N_0-2}}\tau^{m(N_0-1)}(F_{-m+1},\ldots,F_0),\ldots,(F_{-m+1},\ldots,F_0).$$
Note that any fully faithful functor sends the left dual of an exceptional collection to the left dual of its image. 
Hence, by \eqref{la-S-isom-eq}, applying $\la$ to the above collection we get
$$\SS_\CC^{N_0}[N_0(1-d)](F_{-r+1},\ldots,F_0),\ldots,\SS_{\CC}^{N_0-1}[(N_0-1)(1-d)](F_{-m+1},\ldots,F_0),\ldots,F_{-m+1},\ldots,F_0$$
which is the part of the helix generated by $F_{-m+1},\ldots,F_0$ (up to shifts).

We also see from above that the map on directed $\Ext$'s (from left to right) of this collection, induced by $\la$, is an isomorphism.
\ed

\subsubsection{Recursion for categories of matrix factorizations}

Finally, we can prove the directed $A_\infty$-category $AT(a)$ is obtained from $AT(a-)$ by the recursion $\RR$
with $N=\mu^\vee(a)$.

\begin{thm}\label{B-recursion-thm} 
Let us start with the AT exceptional collection $E,E(1),\ldots,E(\mu^\vee(a-)-1)$
in $\MF_{\Ga_{a-}}(p_{a-})$,
extend it to a helix and take the segment $H_{-\mu^\vee(a)+1},\ldots, H_{-1},H_0$ such that $H_0=E$.
Now take the directed $A_\infty$-subcategory with the objects
$F_{-\mu^\vee(a)+1},\ldots, F_{-1},F_0$, where 
$$F_{-i}:=H_i[-\lfloor \frac{i}{m}\rfloor(n+2m(a-)-1)],$$
where $m(a-)=m(a_1,\ldots,a_{n-1})$ is determined by \eqref{Serre-shift-ma-eq}.
Then the directed $A_\infty$-category corresponding to the dual right exceptional collection to $F_{-\mu^\vee(a)+1},\ldots, F_{-1},F$ 
is equivalent to $AT(a)$.
\end{thm}

\Pf . Using Theorem \ref{emb-dual-thm} and Proposition \ref{recursion-Bside-prop} we get the statement with
$$F_{-i}:=H_i[-\lfloor \frac{i}{m}\rfloor(D(a)-1)],$$
where $D(a)$ is the unique integer such that $\Hom^{D(a)}(E,E(\mu^\vee(a-)))\neq 0$ in $\MF_{\Ga_a}(p_a)$.
Note that Equation \eqref{tau-commutativity} holds in our case, 
due to the commutativity of the $\Ext$-algebra of the AT collection. The perfect pairing property also follows from
the structure of the $\Ext$-algebra (see Sec.\ \ref{Ext-sec}).

It remains to check the equality
$$D(a)=n+2m(a-).$$
To this end we observe that $D(a)$ is determined as follows. If $n$ is even then we should have
$$(a_1-1)\ov{x}_1+(a_3-1)\ov{x}_3+\ldots+(a_{n-1}-1)\ov{x}_{n-1}=\mu^\vee(a-)\ov{x}_1+D(a)T.$$
If $n$ is odd then we have
$$T-\ov{x}_1+(a_2-1)\ov{x}_2+(a_4-1)\ov{x}_4+\ldots+(a_{n-1}-1)\ov{x}_{n-1}=-\mu^\vee(a-)\ov{x}_1+D(a)T.$$
But using the relations $a_i\ov{x}_i=2T-\ov{x}_{i+1}$, we immediately see that in both cases the left-hand side is equal
to $\ell_S=nT-\ov{x}_1-\ldots-\ov{x}_n$. Hence, the assertion follows from Lemma \ref{Serre-Bside-lem}.
\ed
\appendix
\section{Mathematica code}\label{appa}
We provide a simple Mathematica code in Figure \ref{code} for the readers who want to experiment with the results in Section \ref{van-sph-sec}. We stress that we do not use such numerical approximations in our argument. We did use this experimentation to come up with the arguments.

\begin{figure}
\includegraphics[width=\textwidth]{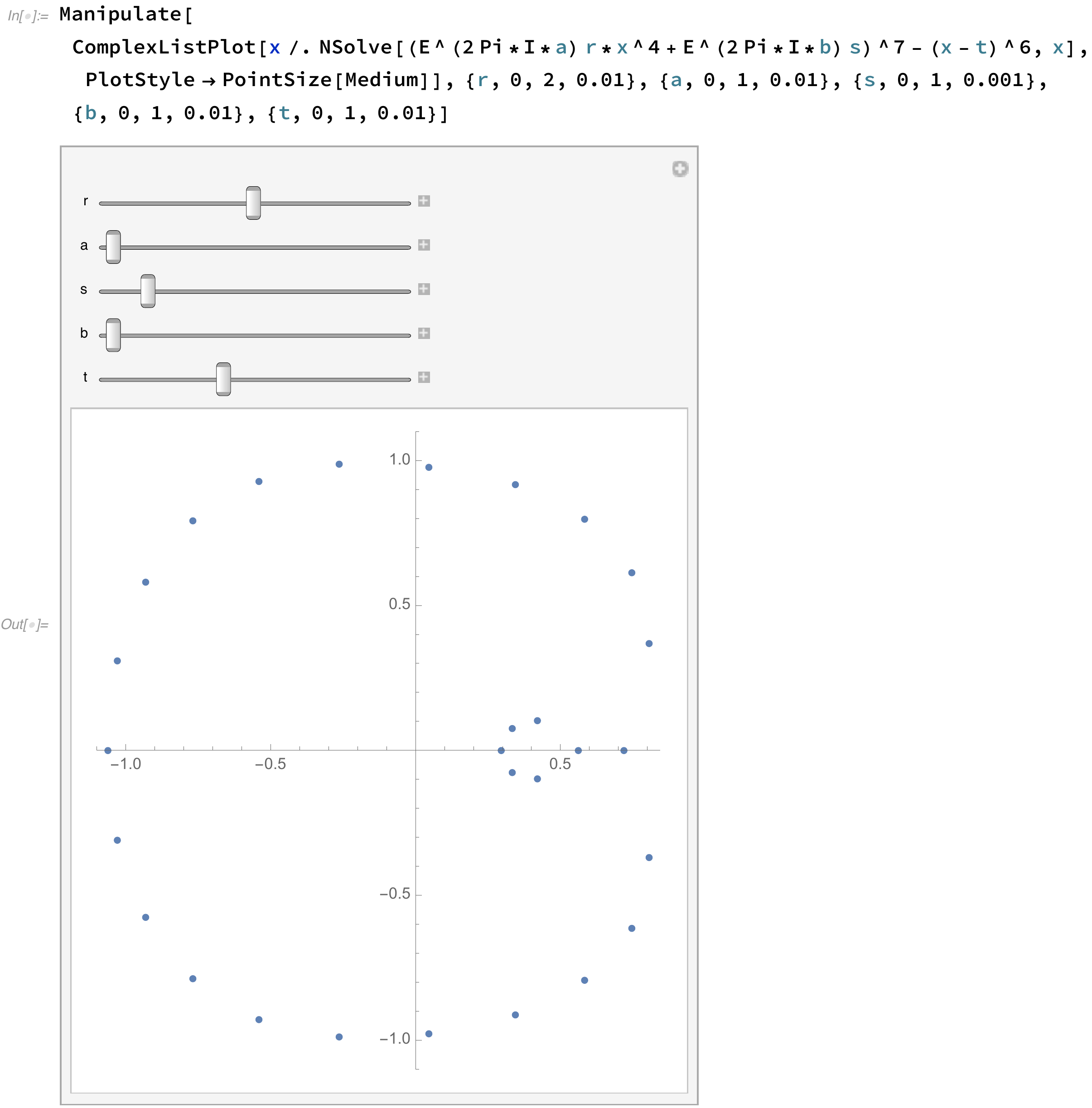}
\label{code}
\end{figure}

\bibliographystyle{plain}
\bibliography{chain}


\end{document}